\documentclass[12pt,reqno]{amsart}
\usepackage[a4paper, textwidth=15cm, textheight=24cm, left=3.5cm, top=3cm, footskip=8mm]{geometry}

\usepackage{graphicx,cite}
\usepackage{mathrsfs,amssymb,amsmath}
\newcommand{\norm}[1]{\left\Vert#1\right\Vert}
\newcommand{\abs}[1]{\left\vert#1\right\vert}

\newcommand{\R}{\mathbb R}
\newcommand{\D}{\partial}
\newcommand{\eps}{\varepsilon}

\newcommand{\dv}{\mathrm{div}\,}
\newcommand{\B}{B}
\newcommand{\vv}{v}

\newcommand{\dm}{{\mathscr{D}_t}}
\newcommand{\idm}{{\mathscr{D}}}
\newcommand{\nb}{\nabla}
\newcommand{\bnb}{\overline{\nb}}
\newcommand{\tr}{\mathrm{tr}\,}
\newcommand{\curl}{\mathrm{curl}\,}
\newcommand{\dist}{\mathrm{dist}\,}
\newcommand{\ls}{\leqslant}
\newcommand{\gs}{\geqslant}
\newcommand{\N}{\mathcal{N}}

\newcommand{\vol}{\mathrm{Vol}\,}
\newcommand{\sgn}{\mathrm{sgn}}
\newcommand{\p}{P }
\newcommand{\bv}{\varsigma}
\newcommand{\no}{\nonumber}
\newcommand{\K}{\mathcal{K}}
\newcommand{\E}{\mathcal{E}}
\newcommand{\T}{\mathcal{T}}

\usepackage{ifpdf}
\ifpdf
\usepackage[CJKbookmarks=true,
         hyperindex=true,
         pdfstartview=FitH,
         bookmarksnumbered=true,
         bookmarksopen=true,
         colorlinks=true,
         citecolor=blue,
         linkcolor=blue,
         urlcolor=blue,
         pdfborder=001,
         pdfauthor={Chengchun Hao},
         pdftitle={},
         pdfkeywords={},
         ]{hyperref}
\else
\usepackage[
         hypertex,
         hyperindex,
         linkcolor=blue,
         unicode,
         citecolor=blue%
         ]{hyperref}
\fi

\allowdisplaybreaks
\numberwithin{equation}{section}
\usepackage{amsthm}
\newtheorem{theorem}{Theorem}[section]

\newtheorem{lemma}[theorem]{Lemma}

\theoremstyle{definition}
\newtheorem{definition}[theorem]{Definition}
\theoremstyle{remark}
\newtheorem{remark}[theorem]{Remark}

\begin{document}

\title[A priori estimates for free boundary incompressible MHD]{A Priori Estimates for\\ Free Boundary Problem of Incompressible\\ Inviscid Magnetohydrodynamic Flows}%

\author{Chengchun Hao}%
\address{Institute of Mathematics,
 Academy of Mathematics \& Systems Science,
 and Hua Loo-Keng Key Laboratory of Mathematics,
  Chinese Academy of Sciences,
   Beijing 100190, China}
\email{hcc@amss.ac.cn}
\thanks{Hao's work was partially supported by the National Natural Science Foundation of China (Grant No. 11171327) and the Youth Innovation Promotion Association, Chinese Academy of Sciences.}
\author{T. Luo}
\address{
Math. Sci. Center,  Tsinghua University,
  and Morningside Center, CAS,
   Beijing, China}
\email{tluo@math.tsinghua.edu.cn}
\date{July 12, 2012}


\begin{abstract}
   In the present paper, we prove the a priori estimates of Sobolev norms for a free boundary problem of the incompressible inviscid MHD equations in all physical spatial dimensions $n=2$ and $3$ by adopting a geometrical point of view used in \cite{CL00}, and estimating quantities such as the second fundamental form and the velocity of the free surface. We identify the well-posedness condition that the outer normal derivative of  the total pressure including the fluid  and magnetic pressures is negative on the free boundary, which is similar to the physical condition (Taylor sign condition) for the incompressible Euler equations of fluids.
\end{abstract}


\maketitle

\tableofcontents
\section{Introduction}

In the present paper, we consider the following incompressible inviscid magnetohydrodynamics (MHD) equations
\begin{subequations}\label{mhd1}
  \begin{align}
    &\vv_t+\vv\cdot\D \vv+\D p=\frac{1}{4\pi}\big(\B\cdot\D \B-\frac{1}{2}\D |\B|^2\big),  \quad \text{ in } \idm, \label{mhd1.1}\\
    &\B_t+\vv\cdot \D \B=\B\cdot \D\vv, \quad \text{ in } \idm \label{mhd1.2}\\
    &\dv \vv=0,  \quad \dv \B=0 , \quad \text{ in } \idm ,\label{mhd1.3}\end{align}
\end{subequations}
describing the motion  of conducting fluids in an electromagnetic field, where the velocity field of the fluids $\vv=(v_1,\cdots,v_n)$ ,  the magnetic field $\B=(B_1,\cdots,B_n)$  and the domain
$\idm\subset [0,  T]\times \R^n$ are the unknowns to be determined. Here $n\in\{2,3\}$ is the spatial dimension, $p$ is the fluid pressure, $\D=(\D_1,\cdots,\D_n)$ and $\dv$ are the usual gradient operator and spatial divergence. Given a simply connected bounded domain $\idm_0\subset \R^n$ and the initial data $\vv_0$ and $\B_0$ satisfying the constraints $ \dv \vv_0=0$ and $\dv \B_0=0$, we want to find a set $\idm\subset [0, T]\times \R^n$ and the vector fields $\vv$ and $\B$ solving \eqref{mhd1} and satisfying the initial conditions:
\begin{equation}\label{initialcondition}
\{x: (0, x)\in \idm\}=\idm_0, \quad  (\vv, \B)=(\vv_0, \B_0)  \text{ on } \{0\}\times \idm_0.
\end{equation}
Let $\idm_t=\{x\in\R^n: (t, x)\in \idm\}$, we also require the following boundary conditions on the free boundary $ \D\dm$:
\begin{subequations}\label{mhd1'}
  \begin{align}
  &\vv_{\N}=\kappa \;\text{ on }   \D\idm_t, \label{mhd1.5'}\\
   &p=0 \;\text{ on }  \D\idm_t, \label{mhd1.5}\\
 &|\B|=\bv \;\text{ and }\;   \B\cdot \N=0 \;\text{ on } \D\idm_t,\label{mhd1.4}
   \end{align}\end{subequations}
for each $t\in [0, T]$, where $\N$ is the exterior unit normal to $\D\dm$,  $\vv_{\N}=\sum_{i=1}^{n}\N^i v_i$, and $\kappa$ is the normal velocity of   $\D\idm_t,$  $\bv$ is a non-negative constant.
Condition \eqref{mhd1.4} should be understood as the constraints on the initial data. Indeed, we will verify that the condition $\B\cdot \N=0$ on $\D\idm_t$
holds for all $t\in [0, T]$ if it holds initially.  We remark here some physical meaning of the boundary conditions. Condition \eqref{mhd1.5'} means that the boundary of $\idm_t$
moves with the fluids, \eqref{mhd1.5} means that outside the fluid region $\idm_t$ is the vacuum, the condition $\B\cdot \N=0 \ \text{ on } \D\idm_t$ comes from the assumption
that the  boundary $\partial\idm$ is a perfect conductor.  Indeed, if we use ${\bf E}$ to denote the electric field induced by the magnetic field $B$, then the boundary condition  $\B\cdot \N=0$ on $\D\idm_t$ gives rise to ${\bf E}\times \N=0$ on $\D\idm_t$. The boundary condition $ |\B|=const$ on $\D\idm_t$ (the magnetic strength is constant on the boundary) is needed to guarantee that the total energy of the system is conserved, i.e.,
$$\frac{d}{dt} \int_{\idm_t}\left( \frac{1}{2}|v|^2+\frac{1}{2}|\B|^2\right)(t,x)dx=0.$$
Condition \eqref{mhd1.4} includes the widely used (e.g., \cite{HW08}) zero magnetic field boundary condition as the special case, but it is much more general and physically
reasonable.

In the  classical plasma-vacuum interface problem (cf. \cite{GP,tr1}), suppose that the interface between the plasma
region $\Omega_{p}(t)$ and the vacuum region $\Omega_v(t)$ is $\Gamma(t)$ which moves with the plasma, then it requires that \eqref{mhd1} holds in the plasma region $\Omega_{p}(t)$, while in the vacuum region $\Omega_v(t)$, the vacuum magnetic field $\mathscr{B}$ satisfies
\begin{equation}\label{vacuum magnetic}
\nabla\times \mathscr{B}=0, \quad \nabla\cdot \mathscr{B}=0.
\end{equation}
On the interface $\Gamma(t)$, it holds that
\begin{equation}\label{interface}
p=0,\quad  |B|=|\mathscr{B}|,\quad  B\cdot \N=\mathscr{B}\cdot \N=0,
\end{equation}
where $\N$ is the unit normal to $\Gamma(t)$. Therefore, the boundary conditions in \eqref{mhd1'} also model the plasma-vacuum
problem for the case when $|\mathscr{B}|$ is constant.

 We will prove a priori bounds for the  free boundary problem \eqref{mhd1}, \eqref{initialcondition} and \eqref{mhd1'} in Sobolev spaces under the following condition
\begin{align}\label{eq.phycond}
  \nb_\N \left(p+\frac{1}{8\pi}|B|^2\right)\ls -\eps<0 \text{ on } \D\dm,
\end{align}
where $\nb_\N=\N^i\D_i$.
We  assume that this condition holds initially, and will verify that it holds true for some time. For the free boundary problem of motion of incompressible fluids in vacuum, without magnetic fields, the natural physical condition (cf. \cite{bealhou,CL00,Coutand,L1,L2,LN,Ebin,SZ,wu1, wu2,zhang}) reads that
\begin{align}\label{eq.phycond00}
  \nb_{\N} p\ls -\eps<0 \text{ on } \D\idm_t,
\end{align}
which excludes the possibility of the Rayleigh-Taylor type instability (see \cite{Ebin}). In this paper,  we
find that the natural physical condition is \eqref{eq.phycond} when the equations of magnetic field couple with the fluids equation. In fact, the quantity $p+\frac{1}{8\pi}|B|^2$,  the total pressure of the system, will play an important role in our analysis. Roughly speaking, the velocity tells the boundary where to move, and the boundary is the level set of the total pressure that determines the acceleration.

The free surface problem of the incompressible Euler equations of fluids has attracted much attention in the recent decades.  Important progress has been made for flows with or without vorticity, with or without surface tension. We refer readers to \cite{AM,CL00,Coutand,L1,L2,LN,Ebin,SZ,wu1,wu2,zhang}.

On the other hand, there have been only few results on the interface problems for the MHD equations. This is due to the difficulties caused by the strong coupling between
the velocity fields and magnetic fields. In this direction,  the well-posedness of a linearized compressible plasma-vacuum interface problem was investigated in \cite{tr1},  and a stationary problem was studied in \cite{FL95}. The current-vortex
sheets problem was studied in \cite{chenwang} and \cite{tr2}. For the incompressible viscous MHD equations, a free boundary problem  in a simply connected domain of $\R^3$ was studied by a linearization technique and the construction of a sequence of successive approximations in \cite{PS10} with an irrotational condition for magnetic fields in a part of the domain.

In this paper, we prove the a priori estimates for the free boundary problem \eqref{mhd1}, \eqref{initialcondition} and \eqref{mhd1'} in all physical spatial dimensions $n=2,3$ by adopting a geometrical point of view used in \cite{CL00}, and estimating quantities such as the second fundamental form and the velocity of the free surface. Throughout the paper, we use the Einstein summation convention, that is, when an index variable appears twice in a single term it implies summation of that term over all the values of the index. Denote the material derivative $D_t=\D_t+\vv\cdot\D$ and the total pressure $P=p+\frac{1}{8\pi}|B|^2$, we can write the free boundary problem as
\begin{subequations}\label{mhd3}
\begin{align}
    &D_t v_j+\D_j\p =\frac{1}{4\pi}B^k\D_k B_j\; \text{ in } \idm, \quad \label{mhd3.1}\\
    &D_t B_j=B^k\D_k v_j\; \text{ in } \idm, \quad \label{mhd3.2}\\
    &\D_jv^j=0 \text{ in } \idm; \quad \D_jB^j=0 \text{ on } \{t=0\}\times \idm_0,\label{mhd3.3}\\
    &v_{\N}=\kappa\; \text{ on } [0,T]\times\D\idm_t, \label{mhd3.11}\\
    &|\B|=\bv \;\text{ on } \D\idm,  \quad\B_j \N^j=0\; \text{ on } \{t=0\}\times\D\idm_0,\label{mhd3.4}\\
    &p=0 \;\text{ on } \D\idm,\label{mhd3.5}\\
    & \nabla_{\N}P<0 \;\text{ on } \{t=0\}\times\D\idm_0.\label{mhd3.5'}
\end{align}
\end{subequations}
We will derive the energy estimates from which the Sobolev norms of $H^s(\idm_t)$ ($s\ls n+1$) of solutions will be derived. For this purpose, we define the energy norms as follows:
The zeroth-order energy, $E_0(t)$, is defined as  the total energy of the system,  i.e.,
\begin{equation}\label{een0}
E_0(t)=\int_{\idm_t} \delta^{ij}(v_iv_j+\frac{1}{8\pi}B_iB_j)dx,\end{equation}
which is conserved,
i.e.,
\begin{equation}\label{een0}
E_0(t)=E_0(0), \qquad \text{ for } 0\ls t\ls T.
\end{equation}
The higher order energy norm has a boundary part and an interior part. The boundary part controls the norms of the second fundamental form of the free surface, the interior part controls the norms of the velocity, magnetic fields and hence the pressure. We will prove that the time derivatives of the energy norms are controlled by themselves. A crucial point in the construction of the higher order energy norms is that the time derivatives of the interior parts will, after integrating by parts, contribute some boundary terms that cancel
the leading-order  terms in the corresponding time derivatives of the boundary integrals.   To this end,  we need to project the equations  for the total pressure $P=p+\frac{1}{8\pi}|B|^2$ to the tangent space of the boundary.  The orthogonal projection $\Pi$ to the tangent space of the boundary of a $(0, r)$ tensor $\alpha$ is defined to be the projection of each component along the normal:
\begin{equation}\label{projection}
(\Pi \alpha)_{i_1\cdots i_r}=\Pi_{i_1}^{j_1}\cdots \Pi_{i_r}^{j_r} \alpha_{j_1\cdots j_r},  \quad \text{ where } \Pi_i^j=\delta_i^j-{\N}_i{\N}^j, \end{equation}
with ${\N}^j=\delta^{ij} {\N}_i={\N}_j$.

Let $\bar \partial_i=\Pi_i^j \partial_j$ be a tangential derivative. If $q=const$ on $\partial \idm_t$, it follows that $\bar \partial_i q=0$ there and
\begin{equation}\label{esff1}
(\Pi \partial^2 q)_{ij}=\theta_{ij} \nabla_{\N} q,
\end{equation}
where ${\theta}_{ij}=\bar \partial_i \N_j$ is the second fundamental form of $\partial \idm_t$.
The higher order energies are defined as:
For $r\ge 1$
\begin{align}\label{ereen}
E_r(t)=&\int_{\idm_t}\delta^{ij}\left(Q(\D^r v_i, \D^r v_j)+\frac{1}{4\pi}Q(\D^r B_i, \D^r B_j)\right)dx\notag\\
       &+\int_{\idm_t}\left(|\D^{r-1}\curl v|^2+\frac{1}{4\pi}|\D^{r-1}\curl B|^2\right)dx\notag\\
       &+I(r)\int_{\D\idm_t}Q(\D^r P, \D^r P)\vartheta dS,
\end{align}
where  $I(r)=0$ if $r=1$ and $I(r)=1$ for $r>1$, so we do not need the boundary integral for $r=1$,
$$\vartheta=(-\nabla_{\N} P)^{-1}.$$
Here  $Q$ is a positive definite quadratic form which,
when restricted to the boundary, is the inner product of the tangential components $Q(\alpha, \beta)=\langle\Pi \alpha, \Pi \beta\rangle$  and in the interior $Q(\alpha, \alpha)$ increases to the norm $|\alpha|^2$. To be more specific, let
\begin{equation}\label{quadratic form}
Q(\alpha, \beta)=q^{i_1j_1}\cdots q^{i_rj_r}\alpha_{i_1\cdots i_r}\beta_{j_1\cdots j_r}
\end{equation}
where
\begin{equation}
q^{ij}=\delta^{ij}-\eta(d)^2\N^i\N^j, \quad d(x)=\dist(x, \D\idm_t), \quad \N^i=-\delta^{ij}\D_j d. \end{equation}
Here  $\eta$ is a smooth cutoff function satisfying $0\ls \eta(d)\ls 1, \quad \eta(d)= 1$ when
$d < d_0/4$ and $\eta(d)=0$ when $d > d_0/2. $ $d_0$  is a fixed number that is smaller than
the injectivity radius of the normal exponential map $\iota_0$, defined to be the largest
number $\iota_0$ such that the map
\begin{equation}
 \D\idm_t\times (-\iota_0, \iota_0)\to \{x\in \R^n:  \dist(x,\  \D\idm_t) < \iota_0\}\end{equation}
 given by
$$(\bar x, \iota)\to x=\bar x+\iota \N(\bar x)$$
is an injection.

The main theorems in this paper are as follows:
 \begin{theorem}\label{thm.1energy'}
  For any smooth solution of the free boundary problem \eqref{mhd3} for $0\ls t\ls T$ satisfying
  \begin{align}
    |\D \p|\ls M, \quad |\D v|\ls& M,  &&\text{in } \idm_t,\\
    |\theta|+|\D v|+\frac{1}{\iota_0}\ls &K,&&\text{on }  \D\idm_t, \label{eq.1energy1.1}
  \end{align}
  we have for $t\in[0,T]$
  \begin{align}
    E_1(t)\ls 2e^{CMt}E_1(0)+C K^2\left(\vol  \idm_t + E_0(0)\right)\left(e^{CMt}-1\right),
  \end{align}
  for some positive constants $C$ and $M$.

\end{theorem}

\begin{theorem}\label{thm.renergy'}
  Let $r\in \{2,\cdots,n+1\}$, then there exists a $T>0$ such that the following holds: For any smooth solution of the free boundary problem \eqref{mhd3} for $0\ls t\ls T$ satisfying
  \begin{align}
   |B|\ls& M_1 \quad\text{for } r=2,&&\text{in } \idm_t,\label{eq.2energy81'}\\
   \quad |\D \p|\ls M, \quad |\D v|\ls& M, \quad |\D B|\ls M,  &&\text{in } \idm_t,\label{eq.2energy8'}\\
    |\theta|+1/\iota_0\ls &K,&&\text{on }  \D\idm_t,\label{eq.2energy9'}\\
    -\nb_\N \p\gs \eps>&0, &&\text{on } \D\idm_t,\label{eq.2energy91'}\\
    |\D^2\p|+|\nb_\N D_t\p|\ls& L,&&\text{on } \D\idm_t,\label{eq.2energy92'}
  \end{align}
  we have, for $t\in[0,T]$,
  \begin{align}\label{eq.renergy'}
  E_r(t)\ls e^{C_1t}E_r(0)+C_2\left(e^{C_1t}-1\right),
\end{align}
where the positive constants $C_1$ and $C_2$ depend on $K$, $K_1$, $M$, $M_1$, $L$, $1/\eps$, $\vol \idm_t$, $E_0(0)$, $E_1(0)$, $\cdots$, and $E_{r-1}(0)$.
\end{theorem}

Most of the a priori  bounds \eqref{eq.2energy81'}-\eqref{eq.2energy92'}  can be obtained from the energy norms by the elliptic estimates which are used to control all components of $\D^r v$, $\D^r B$ and $\D^r p$ from the tangential components $\Pi \D^r P$ in the energy norms, and
 a bound for the second fundamental form of the free boundary
$$\|\bar \D^{r-2}\theta\|_{L^2(\D\idm_t)}\ls C\left( K, L, M, \frac{1}{\epsilon}, E_{r-1}, \vol\idm_t\right)E_r$$
for $r\ge 2$, which controls the regularity of the free boundary.

Since $E_0(t)=E_0(0)$ and $\vol \idm_t=\vol \idm_0$, recursively we can prove the following main theorem from Theorems \ref{thm.1energy'}-\ref{thm.renergy'}.

\begin{theorem}
Let \begin{align}
  \K(0)=&\max\left(\norm{\theta(0,\cdot)}_{L^\infty(\D\idm_0)}, 1/\iota_0(0)\right),\label{eq.K'}\\
  \E(0)=&\norm{1/(\nb_N P(0,\cdot))}_{L^\infty(\D\idm_0)}=1/\eps(0)>0. \label{eq.E'}
\end{align}
   There exists a continuous function $\T>0$ such that if
  \begin{align}
    T\ls \T(\K(0),\E(0),E_0(0),\cdots, E_{n+1}(0),\vol \idm_0 ),
  \end{align}
 then any smooth solution of the free boundary problem for MHD equations \eqref{mhd3} for $0\ls t\ls T$ satisfies
  \begin{align}
    \sum_{s=0}^{n+1} E_s(t)\ls 2\sum_{s=0}^{n+1} E_s(0), \quad 0\ls t\ls T.
  \end{align}
\end{theorem}

In order to prove the above theorems, we need to use the elliptic estimates of the pressure $p$. However, the time derivative of $\Delta p$  involves a third-order term of the velocity which needs to be controlled by higher order energies. In order to overcome this difficulty, we work on the equations for the total pressure
$\p=p+\frac{1}{8\pi}|\B|^2,$ instead of those for the fluid pressure $p$.

Before we close this introduction, we mention here some studies on viscous or inviscid MHD equations, including the Cauchy problem or initial boundary value problems for the fixed boundaries \cite{DL02,DL72,LW11,HX05,HW08,HW10,PS10,ST83, Sec02, YM} and the references therein.

The rest of this paper is organized as follows: In section 2, we use the Lagrangian coordinates to transform the free boundary problem to a fixed initial boundary problem. The Lagrangian transformation induces a Riemannian metric on $\idm_0$, for which we recall the time evolution properties derived in \cite{CL00}, and prove some new identities which will be used later. We also write the equations in Lagrangian coordinates, by using the covariant spatial derivatives with respect to the Riemannian metric induced by the Lagrangian transformation, instead of using the ordinary derivatives. In section 3, we prove the conservation of the zeroth order energy $E_0(t)$, from which one can see that the boundary conditions on the magnetic fields $B$ is necessary for this energy conservation. We also prove in section 3 that the condition $B\cdot \N=0$ on the boundary propagates along the boundary. Section 4 is devoted to the first order energy estimates. In section 5, we prove the higher order energy estimates, by using the identities derived in section 2, the time evolution property of the metric on the boundary induced by the above mentioned Riemannian metric induced by the Lagrangian transformation, the projection properties and the elliptic estimates. In the derivation of the higher order energy estimates in section 5, some a priori assumptions are made, which will be justified in section 6. We also give an appendix on some estimates used in the previous sections, which are basically proved in \cite{CL00}.

\section{Reformulation in Lagrangian Coordinates}

Assume that we are given a velocity vector field $\vv(t,x)$ defined in a set $\idm\subset [0,T]\times\R^n$ such that the boundary of $\dm=\{x: (t,x)\in\idm\}$ moves with the velocity, i.e., $(1,\vv)\in T(\D\idm)$. We will now introduce Lagrangian or co-moving coordinates, that is, coordinates that are constant along the integral curves of the velocity vector field so that the boundary becomes fixed in these coordinates (cf. \cite{CL00}). Let $x=x(t,y)=f_t(y)$ be the trajectory of the fluid given by
\begin{align}\label{trajectory}
  \left\{\begin{aligned}
    &\frac{dx}{dt}=\vv(t,x(t,y)), \quad (t,y)\in [0,T]\times \Omega,\\
    &x(0,y)=f_0(y), \quad y\in\Omega.
  \end{aligned}\right.
\end{align}
where, when $t=0$, we can start with either the Euclidean coordinates in $\Omega=\idm_0$ or some other coordinates $f_0:\Omega\to \idm_0$ where $f_0$ is a diffeomorphism in which the domain $\Omega$ becomes simple. For each $t$, we will then have a change of coordinates $f_t:\Omega\to \dm$, taking $y\to x(t,y)$.  The Euclidean metric $\delta_{ij}$ in $\dm$ then induces a metric
\begin{align}\label{metric1}
  g_{ab}(t,y)=\delta_{ij}\frac{\D x^i}{\D y^a}\frac{\D x^j}{\D y^b}
\end{align}
and its inverse
\begin{align}
  g^{cd}(t,y)=\delta^{kl}\frac{\D y^c}{\D x^k}\frac{\D y^d}{\D x^l}
\end{align}
in $\Omega$ for each fixed $t$.

We will use covariant differentiation in $\Omega$ with respect to the metric $g_{ab}(t,y)$, since it corresponds to differentiation in $\dm$ under the change of coordinates $\Omega \ni y\to x(t,y)\in\dm$, and we will work in both coordinate systems. This also avoids possible singularities in the change of coordinates. We will denote covariant differentiation in the $y_a$-coordinates by $\nb_a$, $a=0, \cdots, n$, and differentiation in the $x_i$-coordinates by $\D_i$, $i=1,\cdots,n$. The covariant differentiation of a $(0,r)$ tensor $k(t,y)$ is the $(0,r+1)$ tensor given by
\begin{align}
  \nb_a k_{a_1\cdots a_r}=\frac{\D k_{a_1\cdots a_r}}{\D y^a}-\Gamma_{aa_1}^d k_{d\cdots a_r}-\cdots-\Gamma_{aa_r}^dk_{a_1\cdots d},
\end{align}
where the Christoffel symbols $\Gamma_{ab}^d$ are given by
\begin{align}
  \Gamma_{ab}^c=\frac{g^{cd}}{2}\left(\frac{\D g_{bd}}{\D y^a}+\frac{\D g_{ad}}{\D y^b}-\frac{\D g_{ab}}{\D y^d}\right)=\frac{\D y^c}{\D x^i}\frac{\D^2 x^i}{\D y^a \D y^b}.
\end{align}
If $w(t,x)$ is the $(0,r)$ tensor expressed in the $x$-coordinates, then the same tensor $k(t,y)$ expressed in the $y$-coordinates is given by
\begin{align}
  k_{a_1\cdots a_r}(t,y)=\frac{\D x^{i_1}}{\D y^{a_1}}\cdots \frac{\D x^{i_r}}{\D y^{a_r}}w_{i_1\cdots i_r}(t,x), \quad x=x(t,y),
\end{align}
and by the transformation properties for tensors,
\begin{align}\label{eq.covtensor}
  \nb_a k_{a_1\cdots a_r}=\frac{\D x^i}{\D y^a}\frac{\D x^{i_1}}{\D y^{a_1}}\cdots \frac{\D x^{i_r}}{\D y^{a_r}}\frac{\D w_{i_1\cdots i_r}}{\D x^i}.
\end{align}
Covariant differentiation is constructed so the norms of tensors are invariant under changes of coordinates,
\begin{align}\label{eq.norminv}
  g^{a_1b_1}\cdots g^{a_rb_r} k_{a_1\cdots a_r}k_{b_1\cdots b_r}=\delta^{i_1j_1}\cdots \delta^{i_rj_r}w_{i_1\cdots i_r}w_{j_1\cdots j_r}.
\end{align}

Furthermore, expressed in the $y$-coordinates,
\begin{align}\label{eq.Di}
  \D_i=\frac{\D}{\D x^i}=\frac{\D y^a}{\D x^i}\frac{\D}{\D y^a}.
\end{align}
Since the curvature vanishes in the $x$-coordinates, it must do so in the $y$-coordinates, and hence
\begin{align}
  [\nb_a,\nb_b]=0.
\end{align}
Let us introduce the notation ${{k_{a\cdots}}^b}_{\cdots c}=g^{bd}k_{a\cdots d\cdots c}$, and recall that covariant differentiation commutes with lowering and rising indices: $g^{ce}\nb_a k_{b\cdot e\cdots d}=\nb_a g^{ce}k_{b\cdot e\cdots d}$. Let us also introduce a notation for the material derivative
\begin{align}
  D_t=\left.\frac{\D}{\D t}\right|_{y=\textrm{const}}=\left.\frac{\D}{\D t}\right|_{x=\textrm{const}}+v^k\frac{\D}{\D x^k}.
\end{align}
Then we have, from \cite[Lemma 2.2]{CL00}, that
\begin{align}\label{eq.Dt}
  D_tk_{a_1\cdots a_r}=\frac{\D x^{i_1}}{\D y^{a_1}}\cdots \frac{\D x^{i_r}}{\D y^{a_r}}\left(D_t w_{i_1\cdots i_r}+\frac{\D v^\ell}{\D x^{i_1}}w_{\ell\cdots i_r}+\cdots +\frac{\D v^\ell}{\D x^{i_r}}w_{i_1\cdots\ell}\right).
\end{align}

Now we recall a result concerning time derivatives of the change of
coordinates and commutators between time derivatives and space derivatives (cf. \cite[Lemma 2.1]{CL00}).

\begin{lemma}\label{lem.CL00lem2.1}
  Let $x=f_t(y)$ be the change of variables given by \eqref{trajectory}, and let $g_{ab}$ be the metric given by \eqref{metric1}. Let $v_i=\delta_{ij}v^j=v^i$, and set
  \begin{align}
  u_a(t,y)=&v_i(t,x)\frac{\D x^i}{\D y^a},  & u^a=&g^{ab}u_b,\label{eq.CL00lem2.1.1}\\
    h_{ab}=&\frac{1}{2}D_t g_{ab}, & h^{ab}=&g^{ac}h_{cd}g^{db}.\label{eq.CL00lem2.1.2}
  \end{align}
  Then
  \begin{align}\label{eq.CL00lem2.1.3}
    D_t\frac{\D x^i}{\D y^a}=\frac{\D x^k}{\D y^a} \frac{\D v^i}{\D x^k},\quad D_t\frac{\D y^a}{\D x^i}=-\frac{\D y^a}{\D x^k}\frac{\D v^k}{\D x^i},
  \end{align}
  \begin{align}\label{eq.CL00lem2.1.4}
    &D_t g_{ab}=\nb_a u_b+\nb_b u_a, &&D_t g^{ab}=-2h^{ab},  \quad D_t d\mu_g=g^{ab}h_{ab}d\mu_g,
  \end{align}
  \begin{align}\label{eq.CL00lem2.1.5}
    &D_t \Gamma_{ab}^c=\nb_a\nb_b u^c,
  \end{align}
  where $d\mu_g$ is the Riemannian volume element on $\Omega$ in the metric $g$.
\end{lemma}

\begin{proof}
  The proof is the same as that of \cite[Lemma 2.1]{CL00} except that we need to make some modification due to the difference of the definition of $h_{ab}$. Indeed, the proof of \eqref{eq.CL00lem2.1.3}, \eqref{eq.CL00lem2.1.5} and the first part of \eqref{eq.CL00lem2.1.4} is the same as the mentioned. The second part of \eqref{eq.CL00lem2.1.4} follows from \eqref{eq.CL00lem2.1.2} since $0=D_t(g^{ad}g_{dc})=(D_tg^{ad}) g_{dc}+g^{ad}D_t g_{dc}=(D_tg^{ad}) g_{dc}+2g^{ad}h_{dc}$ and then $D_t g^{ab}=(D_tg^{ad})\delta_d^b=(D_tg^{ad}) g_{dc}g^{cb}=-2g^{ad}h_{dc}g^{cb}=-2h^{ab}$. The last part of \eqref{eq.CL00lem2.1.4} follows since in local coordinates $d\mu_g=\sqrt{\det g}dy$ and $D_t(\det g)=(\det g ) g^{ab}D_tg_{ab}$.
\end{proof}

We now recall the estimates of commutators between the material derivative $D_t$ and space derivatives $\D_i$ and covariant derivatives $\nb_a$.

\begin{lemma}[\mbox{\cite[Lemma 2.3]{CL00}}]
  Let $\D_i$ be given by \eqref{eq.Di}. Then
  \begin{align}\label{eq.DtDi}
    [D_t,\D_i]=-(\D_i v^k)\D_k.
  \end{align}
  Furthermore,
  \begin{align}\label{eq.DtDir}
    [D_t,\D^r]=\sum_{s=0}^{r-1}-\left(r\atop s+1\right)(\D^{1+s}v)\cdot \D^{r-s},
  \end{align}
  where the symmetric dot product is defined to be in components
  \begin{align}
    \left((\D^{1+s}v)\cdot \D^{r-s}\right)_{i_1\cdots i_r}=\frac{1}{r!}\sum_{\sigma\in \Sigma_r}\left(\D_{i_{\sigma_1}\cdots i_{\sigma_{1+s}}}^{1+s} v^k\right)\D_{ki_{\sigma_{s+2}}\cdots i_{\sigma_r}}^{r-s},
  \end{align}
  and $\sum_r$ denotes the collection of all permutations of $\{1,2,\cdots, r\}$.
\end{lemma}

\begin{lemma}[cf. \mbox{\cite[Lemma 2.4]{CL00}}] \label{lem.CL00lem2.4}
  Let $T_{a_1\cdots a_r}$ be a $(0,r)$ tensor. We have
  \begin{align}\label{eq.CL00lem2.4.1}
    [D_t,\nb_a]T_{a_1\cdots a_r}=-(\nb_{a_1}\nb_a u^d)T_{da_2\cdots a_r}-\cdots -(\nb_{a_r}\nb_a u^d)T_{a_1\cdots a_{r-1}d}.
  \end{align}
  If $\Delta=g^{cd}\nb_c\nb_d$ and $q$ is a function, we have
  \begin{align}
    [D_t,g^{ab}\nb_a]T_b=&-2h^{ab}\nb_a T_b-(\Delta u^e)T_e,\label{eq.CL002.4.2}\\
    [D_t,\nb]q=&0, \label{eq.Dtpcommu}\\
    [D_t,\Delta]q=&-2h^{ab}\nb_a\nb_b q-(\Delta u^e)\nb_e q.\label{eq.CL002.4.3}
  \end{align}
  Furthermore,
  \begin{align}\label{eq.commutator}
    [D_t,\nb^r]q=\sum_{s=1}^{r-1}-\left(r\atop s+1\right)(\nb^{s+1}u)\cdot \nb^{r-s} q,
  \end{align}
  where the symmetric dot product is defined to be in components
  \begin{align}
    \left((\nb^{s+1}u)\cdot \nb^{r-s} q\right)_{a_1\cdots a_r}=\frac{1}{r!}\sum_{\sigma\in\Sigma_r}\left(\nb_{a_{\sigma_1}\cdots a_{\sigma_{s+1}}}^{s+1} u^d\right)\nb_{da_{\sigma_{s+2}}\cdots a_{\sigma_r}}^{r-s}q.
  \end{align}
\end{lemma}
\begin{proof}
  The proof is similar to that of \cite[Lemma 2.4]{CL00}. We only need to verify \eqref{eq.CL002.4.2} and \eqref{eq.CL002.4.3} since they involve the term $D_t g^{ab}$. Now from \eqref{eq.CL00lem2.1.4} and \eqref{eq.CL00lem2.4.1}, it follows that
  \begin{align*}
    [D_t, g^{ab}\nb_a]T_b=&D_t(g^{ab}\nb_a T_b)-g^{ab}\nb_a D_t T_b\\
    =&(D_tg^{ab}) \nb_a T_b+g^{ab} D_t\nb_a T_b-g^{ab}\nb_a D_t T_b\\
    =&-2h^{ab}\nb_a T_b+ g^{ab}[D_t,\nb_a]T_b\\
    =&-2h^{ab}\nb_a T_b-g^{ab}\nb_b\nb_a u^e T_e\\
    =&-2h^{ab}\nb_a T_b-(\Delta u^e) T_e.
  \end{align*}
  From \eqref{eq.Dt} and \eqref{eq.DtDi}, we have
  \begin{align*}
    D_t\nb_a q=&D_t\left(\frac{\D x^i}{\D y^a} \D_i q\right)=\frac{\D x^i}{\D y^a}\left(D_t \D_iq+\D_\ell\frac{\D v^\ell}{\D x^i}\right)\\
    =&\frac{\D x^i}{\D y^a}\left([D_t,\D_i]q+ \D_iD_tq+\D_\ell q\frac{\D v^\ell}{\D x^i}\right)\\
    =&\frac{\D x^i}{\D y^a}\left(-\D_i v^k\D_k q+ \D_iD_tq+\D_i v^\ell\D_\ell q\right)=\frac{\D x^i}{\D y^a}\D_iD_tq=\nb_a D_t q,
  \end{align*}
  namely, \eqref{eq.Dtpcommu} follows.  Then, \eqref{eq.CL002.4.3} follows from \eqref{eq.CL002.4.2} and
  \begin{align*}
    [D_t,\Delta]q=&D_t\Delta q-\Delta D_t q=D_t(g^{ab}\nb_a\nb_b q)-g^{ab}\nb_a\nb_b D_tq\\
    =&[D_t,g^{ab}\nb_a]\nb_b q+g^{ab}\nb_a[D_t,\nb_b]q\\
    =&[D_t,g^{ab}\nb_a]\nb_b q.
  \end{align*}
  Therefore, we complete the proof.
\end{proof}

Denote
\begin{align}
  &B_i=\delta_{ij}B^j=B^i, \quad \beta_a=B_j\frac{\D x^j}{\D y^a},\quad \beta^a=g^{ab}\beta_b, \text{ and }
   |\beta|^2=\beta_a\beta^a.
\end{align}
It follows, from \eqref{eq.norminv}, that
\begin{align}\label{eq.beta}
  |\beta|^2=|\B|^2, \quad B_j=\frac{\D y^a}{\D x^j}\beta_a, \quad\p=p+\frac{1}{8\pi}|\beta|^2.
\end{align}
Then $P=\frac{1}{8\pi} \bv^2$ on the boundary $\D\Omega$.

From \eqref{eq.CL00lem2.1.1}, $\eqref{mhd3.1}$, \eqref{eq.beta}, \eqref{eq.CL00lem2.1.3}, \eqref{eq.covtensor}, we have
\begin{align*}
D_t u_a=&D_t\left(v_j\frac{\D x^j}{\D y^a}\right) =\frac{\D x^j}{\D y^a}D_t v_j+v_jD_t\frac{\D x^j}{\D y^a}\\
  =&\frac{\D x^j}{\D y^a}\left(-\D_j\p +\frac{1}{4\pi}B^k\D_k B_j\right)+v_j \frac{\D x^k}{\D y^a} \frac{\D v^j}{\D x^k}\\
  =&-\nb_a\p +\frac{1}{4\pi}\frac{\D x^j}{\D y^a}\delta^{ki}\frac{\D y^b}{\D x^i}\beta_b\delta_k^l \frac{\D y^d}{\D x^l} \frac{\D y^c}{\D x^j}\nb_d\beta_c\\
  &\qquad+\frac{\D y^b}{\D x^j} u_b\delta^{lj}\frac{\D y^c}{\D x^l} \nb_a u_c\\
  =&-\nb_a\p +\frac{1}{4\pi}g^{bd}g_{ae}g^{ec}\beta_b\nb_d\beta_c+g^{bc} u_b\nb_a u_c\\
  =&-\nb_a\p +\frac{1}{4\pi}\beta^d\nb_d\beta_a+u^c\nb_a u_c.
\end{align*}
Similarly, we get
\begin{align*}
  D_t\beta_a=&\frac{\D x^j}{\D y^a}D_t B_j+B_jD_t\frac{\D x^j}{\D y^a}
  =\frac{\D x^j}{\D y^a}B^k\D_k v_j+B_j\frac{\D x^k}{\D y^a} \frac{\D v^j}{\D x^k}\\
  =&\beta^d\nb_d u_a+\beta^c\nb_a u_c.
\end{align*}

Thus, the system \eqref{mhd1} can be written in the Lagrangian coordinates as
\begin{subequations}\label{mhd41}
\begin{align}
    &D_tu_a+\nb_a\p =u^c\nb_a u_c+\frac{1}{4\pi}\beta^d\nb_d\beta_a,\label{mhd41.1}\\
    &D_t\beta_a=\beta^d\nb_d u_a+\beta^c\nb_a u_c,\label{mhd41.2}\\
    &\nb_a u^a=0 \text{ in } [0,T]\times\Omega;\quad \nb_a\beta^a=0 \text{ in } \{t=0\}\times \Omega,\label{mhd41.3}\\
    &|\beta|=\bv \text{ and }  \beta_a N^a=0 \quad \text{on } \{t=0\}\times \D\Omega,\label{mhd41.4}\\
    &p=0 \quad \text{on } [0,T]\times\D\Omega.\label{mhd41.5}
\end{align}
\end{subequations}

\section{The Energy Conservation and Some Conserved Quantities}

Firstly, the divergence free property of $\beta$, i.e., $\dv\beta=0$, is preserved for all times under the Lagrangian coordinates or in view of the material derivative, i.e., $D_t\dv\beta=0$. Indeed, from \eqref{eq.CL002.4.2} and Lemma \ref{lem.CL00lem2.1}, the divergence of $\eqref{mhd41.2}$ gives
  \begin{align*}
    &D_t(g^{ab}\nb_b\beta_a)=[D_t,g^{ab}\nb_b]\beta_a+g^{ab}\nb_b D_t \beta_a\\
    =&-2h^{ab}\nb_b\beta_a-(\Delta u^e) \beta_e+g^{ab}\nb_b(\beta^d\nb_d u_a+\beta^c\nb_a u_c)\\
    =&-2h^{ab}\nb_b\beta_a-(\Delta u^e) \beta_e+\nb_b\beta^d \nb_d u^b+\beta^d\nb_d\nb_b u^b\\
    &+g^{ab}\nb_b \beta^c \nb_a u_c+\beta^c\Delta u_c\\
    =&-g^{ac}(\nb_c u_d+\nb_d u_c)g^{db}\nb_b\beta_a+\nb_b\beta^d \nb_d u^b+g^{ab}\nb_b \beta^c \nb_a u_c\\
    =&0.
  \end{align*}

Secondly, we assume that
\begin{align}\label{eq.beta0}
     |\nb u(t,y)|\ls C\quad  \text{on } [0,T]\times\D\Omega,
\end{align}
then that $\beta\cdot N=0$ is preserved for all times $t$ in the lifespan $[0,T]$, that is, we have $\beta\cdot N=0$ on $[0,T]\times\D\Omega$ if $\beta\cdot N=0$ on $\{t=0\}\times\D\Omega$. Indeed, we have, from \eqref{mhd41.2} and Lemmas \ref{lem.CL00lem2.1} and \ref{lem.CL00lem3.9}, that
\begin{align*}
  D_t(\beta_aN^a)=&D_t(g^{ab}\beta_a N_b)=N^aD_t\beta_a+\beta_a (D_tg^{ab})N_b+\beta_ag^{ab}D_tN_b\\
  =&N^a(\beta^d\nb_du_a+\beta^d\nb_a u_d)-\nb_cu^b\beta^cN_b-N^d\nb_du^a\beta_a+\beta_ag^{ab}h_{NN}N_b\\
  =&h_{NN}\beta_aN^a,
\end{align*}
which implies, by the Gronwall inequality and  the identity $\big|D_t|f|\big|=\abs{D_t f}$, that
\begin{align}
  |(\beta_aN^a)(t,y)|\ls e^{Ct} |(\beta_aN^a(0,y)|=0.
\end{align}

Thus, in view of the above three preserved quantities, the system \eqref{mhd41}, or \eqref{mhd1}, can be written in the Lagrangian coordinates as
\begin{subequations}\label{mhd4}
\begin{align}
    &D_tu_a+\nb_a\p =u^c\nb_a u_c+\frac{1}{4\pi}\beta^d\nb_d\beta_a,\label{mhd4.1}\\
    &D_t\beta_a=\beta^d\nb_d u_a+\beta^c\nb_a u_c,\label{mhd4.2}\\
    &\nb_a u^a=0,\quad \nb_a\beta^a=0, \text{ in } [0,T]\times \Omega,\label{mhd4.3}\\
    &\p=\frac{1}{8\pi}\bv^2, \quad |\beta|=\bv, \quad \beta\cdot N=0,\quad \text{on } [0,T]\times\D\Omega.\label{mhd4.4}
\end{align}
\end{subequations}

Finally, the energy defined by
\begin{align}
  E_0(t)=\int_\Omega \left(\frac{1}{2}|u|^2+\frac{1}{8\pi}|\beta|^2\right)d\mu_g
\end{align}
is conserved. In fact, by \eqref{eq.CL00lem2.1.4}, \eqref{mhd41}, Gauss' formula and the fact $D_td\mu_g=0$ due to $\dv u=0$, it yields
      \begin{align*}
        &\frac{d}{dt}E_0(t)=\int_\Omega D_t \left(\frac{1}{2}g^{ab}u_a u_b+\frac{1}{8\pi}g^{ab}\beta_a\beta_b\right)d\mu_g\\
        =&\int_\Omega \left(u^a D_t u_a+\frac{1}{4\pi} \beta^a D_t \beta_a\right) d\mu_g\\
        &+\int_\Omega \frac{1}{2}(D_t g^{ab}) \left(u_au_b+\frac{1}{4\pi}\beta_a\beta_b \right)d\mu_g\\
        =&\int_\Omega [-u^a\nb_a\p+u^au^c\nb_a u_c+\frac{1}{4\pi}u^a \beta^d \nb_d \beta_a]d\mu_g\\
        &+\int_\Omega\left(\frac{1}{4\pi}\beta^a \beta^d\nb_d u_a+\frac{1}{4\pi}\beta^a \beta^c\nb_a u_c\right)d\mu_g\\
        &-\int_\Omega h^{ab} \left(u_au_b+\frac{1}{4\pi}\beta_a\beta_b \right)d\mu_g\\
        =&-\int_{\D\Omega} N_au^a\p d\mu_\gamma+\int_{\Omega}  u^a u^c\nb_a u_c d\mu_g+\frac{1}{4\pi}\int_{\D\Omega} N_d\beta^d u^a\beta_a d\mu_\gamma\\
        &+\frac{1}{4\pi}\int_\Omega\beta^a \beta^c\nb_a u_c d\mu_g\\
        &-\frac{1}{2}\int_\Omega g^{ac}(\nb_c u_d+\nb_d u_c)g^{db} \left(u_au_b+\frac{1}{4\pi}\beta_a\beta_b \right)d\mu_g\\
        =&0.
      \end{align*}

\section{The First Order Energy Estimates}

From \eqref{eq.CL00lem2.4.1} and \eqref{mhd4.1}, we have
\begin{align*}
  &D_t(\nb_b u_a)+\nb_b\nb_a\p \\
  =&[D_t,\nb_b]u_a+\nb_b D_t u_a+\nb_b\nb_a\p \\
  =&-(\nb_a\nb_b u^d) u_d+\frac{1}{4\pi}\nb_b(\beta^d\nb_d\beta_a)+\nb_b(u^c\nb_a u_c)\\
  =&-(\nb_a\nb_b u^d) u_d+\frac{1}{4\pi}(\nb_b\beta^d \nb_d\beta_a+\beta^d\nb_b\nb_d\beta_a)+\nb_b u^c\nb_au_c+u^c\nb_b\nb_au_c\\
  =&\nb_b u^c\nb_au_c+\frac{1}{4\pi}(\nb_b\beta^d \nb_d\beta_a+\beta^d\nb_b\nb_d\beta_a).
\end{align*}
From \eqref{eq.CL00lem2.4.1} and \eqref{mhd4.2}, we get
\begin{align*}
D_t&(\nb_b\beta_a)=[D_t,\nb_b]\beta_a+\nb_b D_t \beta_a\\
  =&-(\nb_a\nb_b u^d) \beta_d+\nb_b(\beta^d\nb_d u_a+\beta^c\nb_a u_c)\\
  =&-(\nb_a\nb_b u^d) \beta_d+\nb_b\beta^d\nb_d u_a+\beta^d\nb_b\nb_d u_a+\nb_b\beta^c\nb_a u_c+\beta^c\nb_b\nb_a u_c\\
  =&\nb_b\beta^c(\nb_c u_a+\nb_a u_c)+\beta^d\nb_d\nb_b u_a.
\end{align*}
Thus, we obtain
\begin{align}\label{eq.1energy1}
D_t(\nb_b u_a)+&\nb_b\nb_a\p
=\nb_b u^c\nb_au_c+\frac{1}{4\pi}(\nb_b\beta^d \nb_d\beta_a+\beta^d\nb_b\nb_d\beta_a),\\ \label{eq.1energy2}
D_t(\nb_b\beta_a)=&\nb_b\beta^c(\nb_c u_a+\nb_a u_c)+\beta^d\nb_d\nb_b u_a.
\end{align}

Now, we calculate the material derivative of $g^{bd}\gamma^{ae}\nb_a u_b\nb_e u_d$. From \eqref{eq.CL00lem2.1.4}, \eqref{eq.CL00lem2.1.2}, \eqref{eq.CL00lem3.9.2}, we get
\begin{align} \label{eq.1energy3}
  &D_t(g^{bd}\gamma^{ae}\nb_a u_b\nb_e u_d)\no\\
  =&(D_t g^{bd})\gamma^{ae}\nb_a u_b\nb_e u_d +g^{bd}(D_t\gamma^{ae})\nb_a u_b\nb_e u_d +2g^{bd}\gamma^{ae}(D_t\nb_a u_b)\nb_e u_d\no\\
  =&-2g^{bc}h_{cf}g^{fd}\gamma^{ae}\nb_a u_b\nb_e u_d -2g^{bd}\gamma^{ac}h_{cf}\gamma^{fe}\nb_a u_b\nb_e u_d\no\\
  &-2g^{bd}\gamma^{ae}\nb_e u_d\nb_a\nb_b\p +2g^{bd}\gamma^{ae}\nb_e u_d\nb_a u^c\nb_bu_c\no\\
  &+\frac{1}{2\pi}g^{bd}\gamma^{ae}\nb_e u_d(\nb_a\beta^d \nb_d\beta_b+\beta^d\nb_a\nb_d\beta_b)\no\\
  =&-\gamma^{ae}(\nb_cu_f+\nb_fu_c)\nb_a u^c\nb_e u^f -2\gamma^{ac}\gamma^{fe}(\nb_cu_f+\nb_fu_c)\nb_a u^d\nb_e u_d\no\\
  &-2\gamma^{ae}\nb_e u^b\nb_a\nb_b\p
  +2\gamma^{ae}\nb_e u^b\nb_a u^c\nb_bu_c\no\\
  &+\frac{1}{2\pi}\gamma^{ae}\nb_e u^b(\nb_a\beta^d \nb_d\beta_b+\beta^d\nb_a\nb_d\beta_b)\no\\
  =&-2\gamma^{ae}\nb_cu_f\nb_a u^c\nb_e u^f -4\gamma^{ae}\gamma^{fc}\nb_eu_f\nb_a u^d\nb_c u_d+2\gamma^{ae}\nb_e u^b\nb_a u^c\nb_bu_c\no\\
  &-2\gamma^{ae}\nb_e u^b\nb_a\nb_b\p +\frac{1}{2\pi}\gamma^{ae}\nb_e u^b(\nb_a\beta^d \nb_d\beta_b+\beta^d\nb_a\nb_d\beta_b)\no\\
  =&-4\gamma^{ae}\gamma^{fc}\nb_eu_f\nb_a u^d\nb_c u_d-2\gamma^{ae}\nb_e u^b\nb_a\nb_b\p \no\\
  &+\frac{1}{2\pi}\gamma^{ae}\nb_e u^b(\nb_a\beta^d \nb_d\beta_b+\beta^d\nb_a\nb_d\beta_b).
\end{align}

Similarly, from \eqref{eq.1energy2}, we have
\begin{align}\label{eq.1energy4}
  &D_t(g^{bd}\gamma^{ae}\nb_a \beta_b\nb_e \beta_d)\no\\
  =& -\gamma^{ae}(\nb_cu_f+\nb_fu_c)\nb_a \beta^c\nb_e \beta^f -2\gamma^{ac}\gamma^{fe}(\nb_cu_f+\nb_fu_c)\nb_a \beta^d\nb_e \beta_d\no\\
  &+2\gamma^{ae}\nb_e\beta^b(\nb_a\beta^c\nb_c u_b+\nb_a\beta^c\nb_b u_c+\beta^d\nb_d\nb_a u_b)\no\\
  =&-2\gamma^{ae}\nb_cu_f\nb_a \beta^c\nb_e \beta^f-4\gamma^{ac}\gamma^{fe}\nb_cu_f\nb_a \beta^d\nb_e \beta_d\no\\
  &+2\gamma^{ae}\nb_e\beta^b\nb_a\beta^c\nb_c u_b+2\gamma^{ae}\nb_b u_c\nb_e\beta^b\nb_a\beta^c+2\gamma^{ae}\beta^d\nb_e\beta^b\nb_d\nb_a u_b\no\\
  =&-4\gamma^{ae}\gamma^{fc}\nb_eu_f\nb_a \beta^d\nb_c \beta_d+2\gamma^{ae}\nb_e\beta^b\nb_a\beta^c\nb_c u_b\no\\
  &+2\gamma^{ae}\beta^d\nb_a\beta^b\nb_d\nb_e u_b.
\end{align}
Thus, by combining \eqref{eq.1energy3} with \eqref{eq.1energy4}, we obtain
\begin{align}\label{eq.1energy1.u}
  &D_t\left(g^{bd}\gamma^{ae}\nb_a u_b\nb_e u_d+\frac{1}{4\pi}g^{bd}\gamma^{ae}\nb_a \beta_b\nb_e \beta_d\right)\no\\
  =&-4\gamma^{ae}\gamma^{fc}\nb_eu_f\nb_a u^d\nb_c u_d -\frac{1}{\pi}\gamma^{ae}\gamma^{fc}\nb_eu_f\nb_a \beta^d\nb_c \beta_d \no\\
  &-2\nb_b\left(\gamma^{ae}\nb_e u^b\nb_a\p -\frac{1}{4\pi}\gamma^{ae}\beta^b\nb_e u^d\nb_a\beta_d\right)\no\\
  &+2(\nb_b\gamma^{ae})\left(\nb_e u^b\nb_a\p -\frac{1}{4\pi}\beta^b\nb_e u^d\nb_a\beta_d\right)\no\\
  &+\frac{1}{2\pi}\gamma^{ae}\nb_e u^b\nb_a\beta^d \nb_d\beta_b+\frac{1}{2\pi}\gamma^{ae}\nb_e\beta^b\nb_a\beta^c\nb_c u_b.
\end{align}

Now, we calculate the material derivatives of $|\curl u|^2$ and $|\curl \beta|^2$. We have
\begin{align*}
    D_t|\curl u|^2=&D_t\left(g^{ac}g^{bd}(\curl u)_{ab}(\curl u)_{cd}\right)\\
    =&2(D_t g^{ac})g^{bd}(\curl u)_{ab}(\curl u)_{cd}+4 g^{ac}g^{bd}(D_t\nb_a u_b)(\curl u)_{cd}\\
    =&-2g^{ae}g^{fc}g^{bd}(\nb_eu_f+\nb_f u_e)(\curl u)_{ab}(\curl u)_{cd}\\
    &+4 g^{ac}g^{bd}(\curl u)_{cd}\nb_a u^e\nb_bu_e\\
    &-4 g^{ac}g^{bd}(\curl u)_{cd} \nb_a\nb_b\p \\
    &+\frac{1}{\pi}g^{ac}g^{bd}(\curl u)_{cd}(\nb_a\beta^e \nb_e\beta_b+\beta^e\nb_a\nb_e\beta_b)\\
    =&-4g^{ae}g^{bd}\nb_eu^c(\curl u)_{ab}(\curl u)_{cd}\\
    &+\frac{1}{\pi}g^{ac}(\curl u)_{cd}(\nb_a\beta^e \nb_e\beta^d+\beta^e\nb_a\nb_e\beta^d).
\end{align*}
Similarly,
\begin{align*}
    D_t|\curl \beta|^2=&2(D_t g^{ac})g^{bd}(\curl \beta)_{ab}(\curl \beta)_{cd}+4 g^{ac}g^{bd}(D_t\nb_a \beta_b)(\curl \beta)_{cd}\\
    =&-4g^{ae}g^{bd}\nb_eu^c(\curl \beta)_{ab}(\curl \beta)_{cd}\\
    &+4 g^{ac}g^{bd}(\curl \beta)_{cd}(\nb_a\beta^e(\nb_e u_b+\nb_b u_e)+\beta^e\nb_e\nb_a u_b).
\end{align*}
Thus, we can get
\begin{align}\label{eq.1energy1.curl}
  &D_t(|\curl u|^2+\frac{1}{4\pi}|\curl \beta|^2)\no\\
  =&-4g^{ae}g^{bd}\nb_eu^c(\curl u)_{ab}(\curl u)_{cd}+\frac{1}{\pi}g^{ac}(\curl u)_{cd}\nb_a\beta^e \nb_e\beta^d\no\\
  &-\frac{1}{\pi}g^{ae}g^{bd}\nb_eu^c(\curl \beta)_{ab}(\curl \beta)_{cd}\no\\
  &+\frac{1}{\pi} g^{ac}g^{bd}(\curl \beta)_{cd}\nb_a\beta^e(\nb_e u_b+\nb_b u_e)\no\\
  &+\frac{1}{\pi}\nb_e\left(g^{ac}(\curl u)_{cd}\beta^e\nb_a\beta^d\right).
\end{align}

Define the first order energy as
\begin{align}\label{energy.1}
   E_1(t)=&\int_\Omega \left(g^{bd}\gamma^{ae}\nb_a u_b\nb_e u_d+\frac{1}{4\pi}g^{bd}\gamma^{ae}\nb_a \beta_b\nb_e \beta_d\right)d\mu_g\no\\
   &+\int_\Omega\left(|\curl u|^2+\frac{1}{4\pi}|\curl \beta|^2\right)d\mu_g.
\end{align}

Let us recall the Gauss formula for $\Omega$ and $\D\Omega$:
\begin{align}\label{Gauss}
  \int_\Omega \nb_a w^a d\mu_g=\int_{\D\Omega} N_a w^a d\mu_\gamma, \quad \text{and} \quad \int_{\D\Omega}\bnb_a\bar{f}^a d\mu_\gamma =0
\end{align}
if $\bar{f}$ is tangential to $\D\Omega$ and $N$ is the unit conormal to $\D\Omega$.

Then, we get the following estimates.

\begin{theorem}\label{thm.1energy}
  For any smooth solution of MHD \eqref{mhd4} for $0\ls t\ls T$ satisfying
  \begin{align}
    |\nb \p|\ls M, \quad |\nb u|\ls& M,  &&\text{in } [0,T]\times \Omega,\\
    |\theta|+|\nb u|+\frac{1}{\iota_0}\ls &K,&&\text{on } [0,T]\times \D\Omega. \label{eq.1energy1.1}
  \end{align}
  We have for $t\in[0,T]$
  \begin{align}
    E_1(t)\ls 2e^{CMt}E_1(0)+C K^2\left(\vol\Omega+ E_0(0)\right)\left(e^{CMt}-1\right).
  \end{align}
\end{theorem}

\begin{proof}
  By \eqref{eq.1energy1.u}, \eqref{eq.1energy1.curl} and Gauss' formula, we have
  \begin{align}
    &\frac{d}{dt}E_1(t)=\int_\Omega D_t\left(g^{bd}\gamma^{ae}\nb_a u_b\nb_e u_d+\frac{1}{4\pi}g^{bd}\gamma^{ae}\nb_a \beta_b\nb_e \beta_d\right)d\mu_g\no\\
    &\qquad\qquad\quad+\int_\Omega D_t\left(|\curl u|^2+\frac{1}{4\pi}|\curl \beta|^2\right)d\mu_g\no\\
    &\qquad\qquad\quad+\int_\Omega \left(g^{bd}\gamma^{ae}\nb_a u_b\nb_e u_d+\frac{1}{4\pi}g^{bd}\gamma^{ae}\nb_a \beta_b\nb_e \beta_d\right)\tr h d\mu_g\no\\
   &\qquad\qquad\quad+\int_\Omega \left(|\curl u|^2+\frac{1}{4\pi}|\curl \beta|^2\right)\tr h d\mu_g\no\\
   =&-4\int_\Omega\gamma^{ae}\gamma^{fc}\nb_eu_f\nb_a u^d\nb_c u_d d\mu_g -\frac{1}{\pi}\int_\Omega\gamma^{ae}\gamma^{fc}\nb_eu_f\nb_a \beta^d\nb_c \beta_d  d\mu_g\no\\
  &-2\int_{\D\Omega}N_b\left(\gamma^{ae}\nb_e u^b\nb_a\p  d\mu_g-\frac{1}{4\pi}\gamma^{ae}\beta^b\nb_e u^d\nb_a\beta_d\right) d\mu_\gamma \label{energy.1.2}\\
  &+2\int_\Omega(\nb_b\gamma^{ae})\left(\nb_e u^b\nb_a\p -\frac{1}{4\pi}\beta^b\nb_e u^d\nb_a\beta_d\right) d\mu_g \label{energy.1.3}\\
  &+\frac{1}{2\pi}\int_\Omega\gamma^{ae}\nb_e u^b\nb_a\beta^d \nb_d\beta_b d\mu_g+\frac{1}{2\pi}\int_\Omega\gamma^{ae}\nb_e\beta^b\nb_a\beta^c\nb_c u_b d\mu_g\no\\
  &-4\int_\Omega g^{ae}g^{bd}\nb_eu^c(\curl u)_{ab}(\curl u)_{cd} d\mu_g\no\\
  &+\frac{1}{\pi}\int_\Omega g^{ac}(\curl u)_{cd}\nb_a\beta^e \nb_e\beta^d d\mu_g\no\\
  &-\frac{1}{\pi}\int_\Omega g^{ae}g^{bd}\nb_eu^c(\curl \beta)_{ab}(\curl \beta)_{cd} d\mu_g\no\\
  &+\frac{1}{\pi}\int_\Omega g^{ac}g^{bd}(\curl \beta)_{cd}\nb_a\beta^e(\nb_e u_b+\nb_b u_e) d\mu_g\no\\
  &+\frac{1}{\pi}\int_{\D\Omega} N_e\beta^e g^{ac}(\curl u)_{cd}\nb_a\beta^d  d\mu_\gamma.\label{energy.1.9}\\
    &+\int_\Omega \left(g^{bd}\gamma^{ae}\nb_a u_b\nb_e u_d+\frac{1}{4\pi}g^{bd}\gamma^{ae}\nb_a \beta_b\nb_e \beta_d\right)\tr h d\mu_g\no\\
   &+\int_\Omega \left(|\curl u|^2+\frac{1}{4\pi}|\curl \beta|^2\right)\tr h d\mu_g.\no
  \end{align}

  Since $\p=\frac{1}{8\pi}\bv^2$ on $\D\Omega$, it follows that $\bnb \p=0$, i.e., $\gamma_a^d\nb_d \p=0$, and then $\gamma^{ae}\nb_a \p=g^{ce}\gamma_c^a\nb_a \p=0$ on the boundary $\D\Omega$. In addition,  $\beta\cdot N=0$ on $\D\Omega$. Thus, the integrals in \eqref{energy.1.2} and \eqref{energy.1.9} vanish.

  From \eqref{2ndfundform} and \eqref{gammaauc}, we get
  \begin{align}
    \theta_{ab}=(\delta_a^c-N_aN^c)\nb_c N_b=\nb_a N_b-N_a\nb_N N_b=\nb_a N_b,
  \end{align}
  since in geodesic coordinates $\nb_N N=0$. It follows that
  \begin{align*}
    \nb_b\gamma^{ae}=&\nb_b (g^{ae}-N^aN^e)=-\nb_b(N^aN^e)=-(\nb_b N^a)N^e-(\nb_b N^e)N^a\\
    =&-\theta_b^a N^e-\theta_b^eN^a.
  \end{align*}
  Thus,  by the H\"older inequality, \eqref{eq.1energy1.1} and Lemma \ref{lem.CL00lem5.5}, we get
  \begin{align*}
    |\eqref{energy.1.3}|\ls &CK\left(\norm{\nb u}_{L^2(\Omega)}\norm{\nb \p}_{L^\infty(\Omega)}(\vol\Omega)^{1/2}\right.\\
    &\qquad\qquad\qquad+\left.\norm{\nb u}_{L^\infty(\Omega)}\norm{\beta}_{L^2(\Omega)}\norm{\nb \beta}_{L^2(\Omega)}\right)\\
    \ls& CK M\left((\vol\Omega)^{1/2}+ E_0^{1/2}(0)\right)E_1^{1/2}(t).
  \end{align*}
  For other terms, we can use the H\"older inequality directly. It yields
  \begin{align*}
    \frac{d}{dt}E_1(t)\ls &CK M\left((\vol\Omega)^{1/2}+ E_0^{1/2}(0)\right)E_1^{1/2}(t)\\
    &+C\norm{\nb u}_{L^\infty(\Omega)}\left(\norm{\nb u}_{L^2(\Omega)}^2+\norm{\nb \beta}_{L^2(\Omega)}^2\right.\\
    &\qquad\qquad\qquad\qquad \left.+\norm{\curl u}_{L^2(\Omega)}^2+\norm{\curl\beta}_{L^2(\Omega)}^2\right)\\
    \ls&C K M\left((\vol\Omega)^{1/2}+ E_0^{1/2}(0)\right)E_1^{1/2}(t)
    +CM E_1(t).
  \end{align*}
  From the Gronwall inequality, it follows that
  \begin{align*}
    E_1^{1/2}(t)\ls e^{CMt/2}E_1^{1/2}(0)+C K\left((\vol\Omega)^{1/2}+ E_0^{1/2}(0)\right)\left(e^{CMt/2}-1\right),
  \end{align*}
  which implies the desired result.
\end{proof}

\begin{remark}
  Since \eqref{energy.1.2}, especially the integral involving $\p$, vanishes, we do not need the boundary integral in the first order energy $E_1(t)$. But in higher order energies estimates, we need to introduce boundary integrals for $\p$ in order to absorb the analogy integral to \eqref{energy.1.2}.
\end{remark}

\section{The General $r$-th Order Energy Estimates}

From \eqref{eq.Dt}, \eqref{eq.DtDir}, $\eqref{mhd3.1}$, we get
\begin{align*}
  D_t\nb^r u_a=&D_t\nb_{a_1}\cdots \nb_{a_r}u_a
  =D_t\left(\frac{\D x^{i_1}}{\D y^{a_1}}\cdots \frac{\D x^{i_r}}{\D y^{a_r}}\frac{\D x^{i}}{\D y^{a}}\D_{i_1}\cdots \D_{i_r} v_i\right)\\
  =&\frac{\D x^{i_1}}{\D y^{a_1}}\cdots \frac{\D x^{i_r}}{\D y^{a_r}}\frac{\D x^{i}}{\D y^{a}}\left(D_t\D_{i_1}\cdots \D_{i_r} v_i+\frac{\D v^\ell}{\D x^{i_1}}\D_\ell\cdots \D_{i_r} v_i+\cdots\right.\\
  &\qquad\qquad\qquad\quad\qquad\left.+\frac{\D v^\ell}{\D x^{i_r}} \D_{i_1}\cdots \D_\ell v_i+\frac{\D v^\ell}{\D x^i}\D_{i_1}\cdots \D_{i_r}v_\ell\right)\\
  =&\frac{\D x^{i_1}}{\D y^{a_1}}\cdots \frac{\D x^{i_r}}{\D y^{a_r}}\frac{\D x^{i}}{\D y^{a}}\left([D_t,\D^r]v_i+\D^r D_t v_i+\frac{\D v^\ell}{\D x^{i_1}}\D_\ell\cdots \D_{i_r} v_i+\cdots\right.\\
  &\qquad\qquad\qquad\quad\qquad\left.+\frac{\D v^\ell}{\D x^{i_r}} \D_{i_1}\cdots \D_\ell v_i+\frac{\D v^\ell}{\D x^i}\D_{i_1}\cdots \D_{i_r}v_\ell\right)\\
  =&\frac{\D x^{i_1}}{\D y^{a_1}}\cdots \frac{\D x^{i_r}}{\D y^{a_r}}\frac{\D x^{i}}{\D y^{a}}\left(\sum_{s=0}^{r-1}-\left(r\atop s+1\right)(\D^{1+s}v)\cdot \D^{r-s}v_i-\D^r\D_i \p\right.\\
  &\qquad\qquad\qquad\quad\qquad +\frac{1}{4\pi}\D^r(B^k\D_k B_i)+\frac{\D v^\ell}{\D x^{i_1}}\D_\ell\cdots \D_{i_r} v_i+\cdots\\
  &\qquad\qquad\qquad\quad\qquad\left.+\frac{\D v^\ell}{\D x^{i_r}} \D_{i_1}\cdots \D_\ell v_i+\frac{\D v^\ell}{\D x^i}\D_{i_1}\cdots \D_{i_r}v_\ell\right)\\
  =&-\nb^r\nb_a\p -\sum_{s=1}^{r-1}\left(r\atop s+1\right)(\nb^{1+s}u)\cdot \nb^{r-s}u_a\\
  &+\nb_a u^c\nb^r u_c+\frac{1}{4\pi}\sum_{s=0}^r\left(r\atop s\right) \nb^s\beta^c\nb^{r-s}\nb_c \beta_a,
\end{align*}
where
\begin{align}
  \left(\nb^s\beta^c\nb^{r-s} \nb_c\beta_a\right)_{a_1\cdots a_r}=\sum_{\Sigma_r}\nb_{a_{\sigma_1}\cdots a_{\sigma_s}}^s \beta^c \nb_{a_{\sigma_{s+1}}\cdots a_{\sigma_r}}^{r-s}\nb_c\beta_a.
\end{align}
Thus, due to $\dv \beta=0$, we get for $r\gs 2$
\begin{align}\label{eq.r.u}
  &D_t\nb^r u_a+\nb^r\nb_a\p \no\\
  =&(\curl u)_{ac}\nb^r u^c+\sgn(2-r)\sum_{s=1}^{r-2}\left(r\atop s+1\right)(\nb^{1+s}u)\cdot \nb^{r-s}u_a\no\\
  &\qquad+\frac{1}{4\pi}\nb_c\left(\beta^c\nb^r\beta_a\right)+\frac{1}{4\pi}\sum_{s=1}^r\left(r\atop s\right) \nb^s\beta^c\nb^{r-s} \nb_c\beta_a,
\end{align}
where $\sgn(s)$ is the signum function of the real number $s$, i.e., $\sgn(s)=1$ for $s>0$, $\sgn(s)=0$ for $s=0$, and $\sgn(s)=-1$ for $s<0$. Of course, we use this notation $\sgn(2-r)$ to indicate that the related term vanishes for $r=2$.

Similarly, by noticing that $\dv\beta=0$, we have
\begin{align}\label{eq.nb.beta}
 D_t\nb^r \beta_a=&\nb_au_c\nb^r\beta^c-\nb^r u^c\nb_c\beta_a\no\\
 &-\sgn(2-r)\sum_{s=1}^{r-2}\left(r\atop s+1\right)(\nb^{1+s}u)\cdot \nb^{r-s}\beta_a\no\\
  &+\nb_c\left(\beta^c\nb^r u_a\right)+\sum_{s=1}^{r}\left(r\atop s\right) \nb^s\beta^c\nb^{r-s} \nb_c u_a.
\end{align}

Define the $r$-th order energy for $r\gs 2$ as
\begin{align*}
  E_r(t)=&\int_\Omega g^{bd}\gamma^{af}\gamma^{AF}\nb_A^{r-1}\nb_a u_b\nb_F^{r-1}\nb_f u_d d\mu_g+\int_\Omega |\nb^{r-1}\curl u|^2 d\mu_g\\
  &+\frac{1}{4\pi}\int_\Omega g^{bd}\gamma^{af}\gamma^{AF}\nb_A^{r-1}\nb_a \beta_b\nb_F^{r-1}\nb_f \beta_d d\mu_g\\
  &+\frac{1}{4\pi}\int_\Omega |\nb^{r-1}\curl \beta|^2 d\mu_g+\int_{\D\Omega} \gamma^{af}\gamma^{AF}\nb_A^{r-1}\nb_a \p\nb_F^{r-1}\nb_f \p\, \vartheta d\mu_\gamma,
\end{align*}
where $\vartheta=1/(-\nb_N \p)$ as before.

\begin{theorem}\label{thm.renergy}
  Let $r\in \{2,\cdots,n+1\}$, then there exists a $T>0$ such that the following holds: For any smooth solution of MHD \eqref{mhd4} for $0\ls t\ls T$ satisfying
  \begin{align}
   |\beta|\ls& M_1 \quad\text{for } r=2,&&\text{in } [0,T]\times \Omega,\label{eq.2energy81}\\
   \quad |\nb \p|\ls M, \quad |\nb u|\ls& M, \quad |\nb \beta|\ls M,  &&\text{in } [0,T]\times \Omega,\label{eq.2energy8}\\
    |\theta|+1/\iota_0\ls &K,&&\text{on } [0,T]\times \D\Omega,\label{eq.2energy9}\\
    -\nb_N \p\gs \eps>&0, &&\text{on } [0,T]\times \D\Omega,\label{eq.2energy91}\\
    |\nb^2\p|+|\nb_ND_t\p|\ls& L,&&\text{on } [0,T]\times \D\Omega,\label{eq.2energy92}
  \end{align}
  we have, for $t\in[0,T]$,
  \begin{align}\label{eq.renergy}
  E_r(t)\ls e^{C_1t}E_r(0)+C_2\left(e^{C_1t}-1\right),
\end{align}
where $C_1$ and $C_2$ depend on $K$, $K_1$, $M$, $M_1$, $L$, $1/\eps$, $\vol\Omega$, $E_0(0)$, $E_1(0)$, $\cdots$, and $E_{r-1}(0)$.
\end{theorem}

\begin{proof}
  \begin{align}
    &\frac{d}{dt}E_r(t)=\no\\
    &\quad\int_\Omega D_t\left(g^{bd}\gamma^{af}\gamma^{AF}\nb_A^{r-1}\nb_a u_b\nb_F^{r-1}\nb_f u_d\right) d\mu_g\label{eq.r.e1}\\
  &+\frac{1}{4\pi}\int_\Omega D_t\left(g^{bd}\gamma^{af}\gamma^{AF}\nb_A^{r-1}\nb_a \beta_b\nb_F^{r-1}\nb_f \beta_d\right) d\mu_g\label{eq.r.e2}\\
    &+\int_\Omega D_t|\nb^{r-1}\curl u|^2 d\mu_g
+\frac{1}{4\pi}\int_\Omega D_t|\nb^{r-1}\curl \beta|^2 d\mu_g\label{eq.r.e3}\\
    &+\int_\Omega g^{bd}\gamma^{af}\gamma^{AF}\nb_A^{r-1}\nb_a u_b\nb_F^{r-1}\nb_f u_d \tr h d\mu_g\label{eq.r.e4}\\
    &+\int_\Omega |\nb^{r-1}\curl u|^2 \tr hd\mu_g
  +\frac{1}{4\pi}\int_\Omega |\nb^{r-1}\curl \beta|^2 \tr hd\mu_g\label{eq.r.e5}\\
&+\frac{1}{4\pi}\int_\Omega g^{bd}\gamma^{af}\gamma^{AF}\nb_A^{r-1}\nb_a \beta_b\nb_F^{r-1}\nb_f \beta_d \tr h d\mu_g\label{eq.r.e6}\\
&+\int_{\D\Omega} D_t\left(\gamma^{af}\gamma^{AF}\nb_A^{r-1}\nb_a \p\nb_F^{r-1}\nb_f \p\right)\, \vartheta d\mu_\gamma\label{eq.r.e7}\\
   &+\int_{\D\Omega} \gamma^{af}\gamma^{AF}\nb_A^{r-1}\nb_a \p\nb_F^{r-1}\nb_f \p\left(\frac{\vartheta_t}{\vartheta}+\tr h-h_{NN}\right)\, \vartheta d\mu_\gamma.\label{eq.r.e8}
  \end{align}

We first estimate  \eqref{eq.r.e1}, \eqref{eq.r.e2} and \eqref{eq.r.e7}.  From Lemmas \ref{lem.CL00lem2.1} and \ref{lem.CL00lem3.9}, and \eqref{eq.r.u}, we have
\begin{align*}
  &D_t\left(g^{bd}\gamma^{af}\gamma^{AF}\nb_A^{r-1}\nb_a u_b\nb_F^{r-1}\nb_f u_d\right)\\
  =&(D_tg^{bd})\gamma^{af}\gamma^{AF}\nb_A^{r-1}\nb_a u_b\nb_F^{r-1}\nb_f u_d+rg^{bd}(D_t\gamma^{af})\gamma^{AF}\nb_A^{r-1}\nb_a u_b\nb_F^{r-1}\nb_f u_d\\
  &+2g^{bd}\gamma^{af}\gamma^{AF}D_t(\nb_A^{r-1}\nb_a u_b)\nb_F^{r-1}\nb_f u_d\\
  =&-2\nb_c u_e\gamma^{af}\gamma^{AF}\nb_A^{r-1}\nb_a u^c\nb_F^{r-1}\nb_f u^e-4r\nb_c u_e\gamma^{ac}\gamma^{ef}\gamma^{AF}\nb_A^{r-1}\nb_a u^d\nb_F^{r-1}\nb_f u_d\\
  &-2\gamma^{af}\gamma^{AF}\nb_F^{r-1}\nb_f u^b\nb_A^{r-1}\nb_a\nb_b\p +2\gamma^{af}\gamma^{AF}\nb_F^{r-1}\nb_f u^b(\curl u)_{bc}\nb_A^{r-1}\nb_a u^c\\
  &+2\sgn(2-r)\gamma^{af}\gamma^{AF}\nb_F^{r-1}\nb_f u_d\sum_{s=1}^{r-2}\left(r\atop s+1\right)\left((\nb^{s+1}u)\cdot \nb^{r-s}u^d\right)_{Aa}\\
  &+\frac{1}{2\pi}\gamma^{af}\gamma^{AF}\nb_F^{r-1}\nb_f u_d\nb_c\left(\beta^c\nb_{Aa}^r\beta^d\right)\\
  &+\frac{1}{2\pi}\gamma^{af}\gamma^{AF}\nb_F^{r-1}\nb_f u_d\sum_{s=1}^r\left(r\atop s\right) \left(\nb^s\beta^c\nb^{r-s}\nb_c \beta^d\right)_{Aa}.
\end{align*}
Similarly,
\begin{align*}
  &D_t\left(g^{bd}\gamma^{af}\gamma^{AF}\nb_A^{r-1}\nb_a \beta_b\nb_F^{r-1}\nb_f \beta_d\right)\\
  =&-2\nb_c u_e\gamma^{af}\gamma^{AF}\nb_A^{r-1}\nb_a \beta^c\nb_F^{r-1}\nb_f \beta^e-4r\nb_c u_e\gamma^{ac}\gamma^{ef}\gamma^{AF}\nb_A^{r-1}\nb_a \beta^d\nb_F^{r-1}\nb_f \beta_d\\
  &+2\gamma^{af}\gamma^{AF}\nb_F^{r-1}\nb_f \beta^b \nb_bu_c\nb_A^{r-1}\nb_a\beta^c-2\gamma^{af}\gamma^{AF}\nb_F^{r-1}\nb_f \beta^b\nb_c\beta_b\nb_A^{r-1}\nb_a u^c\\
  &+2\sgn(2-r)\gamma^{af}\gamma^{AF}\nb_F^{r-1}\nb_f \beta^b\sum_{s=1}^{r-2}\left(r\atop s+1\right)\left((\nb^{1+s}u)\cdot \nb^{r-s}\beta_b\right)_{Aa}\\
  &+2\gamma^{af}\gamma^{AF}\nb_F^{r-1}\nb_f \beta^d\nb_c\left(\beta^c\nb_{Aa}^r u_d\right)\\
  &+2\gamma^{af}\gamma^{AF}\nb_F^{r-1}\nb_f \beta^d\sum_{s=1}^{r}\left(r\atop s\right)\left( \nb^s\beta^c\nb^{r-s} \nb_c u_d\right)_{Aa},
\end{align*}
and
\begin{align*}
  &D_t\left(\gamma^{af}\gamma^{AF}\nb_A^{r-1}\nb_a \p\nb_F^{r-1}\nb_f \p\right)\\
  =&-4r\nb_c u_e\gamma^{ac}\gamma^{ef}\gamma^{AF}\nb_A^{r-1}\nb_a \p\nb_F^{r-1}\nb_f \p+2 \gamma^{af}\gamma^{AF}\nb_A^{r-1}\nb_a \p D_t\left(\nb_F^{r-1}\nb_f \p\right).
\end{align*}
Thus, we get
\begin{align}
  &\eqref{eq.r.e1}+\eqref{eq.r.e2}+\eqref{eq.r.e7}\no\\
  \ls& C\left(\norm{\nb u}_{L^\infty(\Omega)}+\norm{\nb \beta}_{L^\infty(\Omega)}\right) E_r(t)\no\\
  &+C E_r^{1/2}(t)\sum_{s=1}^{r-2}\norm{\nb^{s+1} u}_{L^4(\Omega)}\left(\norm{\nb^{r-s} u}_{L^4(\Omega)}+\norm{\nb^{r-s} \beta}_{L^4(\Omega)}\right)\label{eq.r.e111}\\
  &+C E_r^{1/2}(t)\sum_{s=2}^{r-1}\norm{\nb^s \beta}_{L^4(\Omega)}\left(\norm{\nb^{r-s+1}u}_{L^4(\Omega)}+\norm{\nb^{r-s+1} \beta}_{L^4(\Omega)}\right)\label{eq.r.e112}\\
  &+2\int_{\D\Omega}\gamma^{af}\gamma^{AF}\nb_{Aa}^{r}\p \left(D_t\nb_{Ff}^{r} \p-\frac{1}{\vartheta}N_b\nb_{Ff}^{r} u^b\right) \vartheta d\mu_\gamma \label{eq.r.e11}\\
  &+2\int_{\Omega}\nb_b\left(\gamma^{af}\gamma^{AF}\right)\nb_F^{r-1}\nb_f u^b\nb_A^{r-1}\nb_a\p d\mu_g\label{eq.r.e12}\\
  &+\frac{1}{2\pi}\int_{\D\Omega} N_c\gamma^{af}\gamma^{AF}\nb_F^{r-1}\nb_f u_d\beta^c\nb_{Aa}^r\beta^d d\mu_\gamma\label{eq.r.e13}\\
  &-\frac{1}{2\pi}\int_{\Omega} \nb_c\left(\gamma^{af}\gamma^{AF}\right)\nb_F^{r-1}\nb_f u_d\beta^c\nb_{Aa}^r\beta^d d\mu_g.\label{eq.r.e14}
\end{align}
Due to $\beta\cdot N=0$ on $\D\Omega$,  \eqref{eq.r.e13} vanishes. Let $\alpha$ be a $(0,r)$ tensor and $n\in\{2,3\}$. Then from Lemma \ref{lem.CL00lemA.4}, we have, for $\iota_1\gs 1/K_1$, that
\begin{align}\label{eq.CllemmaA.4}
  \norm{\beta}_{L^\infty(\Omega)}\ls C\sum_{0\ls s\ls 2} K_1^{n/2-s} \norm{\nb^s \beta}_{L^2(\Omega)}\ls C(K_1)\sum_{s=0}^2 E_s^{1/2}(t).
\end{align}
Thus, for the last integral, by the H\"older inequality and the assumption \eqref{eq.2energy8}, we have for any $r\gs 3$
\begin{align}
  \eqref{eq.r.e14}\ls &CK\norm{\beta}_{L^\infty(\Omega)}E_r(t)\ls C(K,K_1) \left(\sum_{s=0}^2 E_s^{1/2}(t)\right)E_r(t).
\end{align}
For $r=2$, we have to assume the a priori bound $|\beta|\ls M_1$ on $[0,T]\times \Omega$, i.e., \eqref{eq.2energy81}, in order to get a bound that is linear in the highest-order derivative or energy. Then, we have by \eqref{eq.2energy81}
\begin{align}
  \eqref{eq.r.e14}\ls &CK\norm{\beta}_{L^\infty(\Omega)}E_r(t)\ls C(K,M_1) E_r(t), \text{ for } r=2.
\end{align}

By the H\"older inequality, we have
\begin{align}\label{eq.up}
  \eqref{eq.r.e12}\ls CKE_r^{1/2}(t)\norm{\nb^r \p}_{L^2(\Omega)}.
\end{align}

From \eqref{mhd3.1}, we have
\begin{align*}
  \D_j(D_t v^j)+\Delta\p =\frac{1}{4\pi}\D_j(B^k\D_k B^j),
\end{align*}
which yields from \eqref{eq.DtDi}
\begin{align*}
  \Delta\p =-\D_j v^k\D_k v^j+\frac{1}{4\pi}\D_j B^k\D_k B^j.
\end{align*}
Since $\Delta$ is invariant, we have
\begin{align}\label{eq.2e.p}
  \Delta \p=-\nb_a u^b\nb_b u^a+\frac{1}{4\pi}\nb_a \beta^b \nb_b \beta^a.
\end{align}

It follows that for $r\gs 2$
\begin{align*}
  \nb^{r-2}\Delta \p=&\nb^{r-2}\left(-\nb_a u^b\nb_b u^a+\frac{1}{4\pi}\nb_a \beta^b \nb_b \beta^a\right)\\
  =&-\sum_{s=0}^{r-2}\left(r-2\atop s\right) \nb^s\nb_a u^b\nb^{r-2-s}\nb_b u^a\\
  &+\frac{1}{4\pi}\sum_{s=0}^{r-2}\left(r-2\atop s\right)\nb^s\nb_a \beta^b\nb^{r-2-s}\nb_b \beta^a.
\end{align*}
From \eqref{eq.CllemmaA.4},  we have for $s\gs 0$
\begin{align}\label{eq.r.einfbeta}
  \norm{\nb^s \beta}_{L^\infty(\Omega)}\ls &C\sum_{\ell=0}^{2} K_1^{n/2-\ell}\norm{\nb^{\ell+s}\beta}_{L^2(\Omega)}
  \ls C(K_1)\sum_{\ell=0}^{2} E_{s+\ell}^{1/2}(t),
\end{align}
and, similarly,
\begin{align}\label{eq.r.einfu}
  \norm{\nb^s u}_{L^\infty(\Omega)}\ls C(K_1)\sum_{\ell=0}^{2} E_{s+\ell}^{1/2}(t).
\end{align}
By H\"older's inequality, \eqref{eq.r.einfbeta} and \eqref{eq.r.einfu}, we get for $r\in\{3,4\}$,
\begin{align}\label{eq.deltap34}
  &\norm{\nb^{r-2}\Delta \p}_{L^2(\Omega)}\no\\
  \ls &C\sum_{s=0}^{r-2} \norm{\nb^s\nb_a u^b\nb^{r-2-s}\nb_b u^a}_{L^2(\Omega)}+C\sum_{s=0}^{r-2}\norm{\nb^s\nb_a \beta^b\nb^{r-2-s}\nb_b \beta^a}_{L^2(\Omega)}\no\\
  \ls&C\norm{\nb u}_{L^\infty(\Omega)}\norm{\nb^{r-1} u}_{L^2(\Omega)}+C\norm{\nb \beta}_{L^\infty(\Omega)}\norm{\nb^{r-1} \beta}_{L^2(\Omega)}\no\\
  &+(r-3)C\left(\norm{\nb^{2} u}_{L^\infty(\Omega)}\norm{\nb^{2} u}_{L^2(\Omega)}+\norm{\nb^{2} \beta}_{L^\infty(\Omega)}\norm{\nb^{2} \beta}_{L^2(\Omega)}\right)\no\\
  \ls& C(K_1)E_{r-1}^{1/2}(t)\sum_{\ell=1}^{3} E_{\ell}^{1/2}(t)+(r-3)C(K_1) E_{2}^{1/2}(t)\sum_{\ell=2}^{4} E_{\ell}^{1/2}(t)\no\\
  \ls&C(K_1)\sum_{\ell=1}^{r-1} E_\ell(t)+C(K_1)E_2^{1/2}(t)E_r^{1/2}(t).
\end{align}
For $r=2$, we have a simple estimate from the assumption \eqref{eq.2energy8} and H\"older's inequality, i.e.,
\begin{align}\label{eq.deltap2}
  \norm{\Delta\p}_{L^2(\Omega)}\ls&C\norm{\nb u}_{L^2(\Omega)}\norm{\nb u}_{L^\infty(\Omega)}+C\norm{\nb \beta}_{L^2(\Omega)}\norm{\nb \beta}_{L^\infty(\Omega)}\no\\
  \ls& CME_1^{1/2}(t),
\end{align}
which is a lower energy term.
Thus, by \eqref{eq.CL00prop5.8.2}, \eqref{eq.deltap34} and \eqref{eq.deltap2}, we obtain for any $\delta_r>0$
\begin{align}\label{eq.est.nbrp}
\norm{\nb^r \p}_{L^2(\Omega)}
  \ls &\delta_r\norm{\Pi \nb^r \p}_{L^2(\D\Omega)}\no\\
  &+C(1/\delta_r,K,\vol\Omega)\sum_{s\ls r-2}\norm{\nb^s\Delta \p}_{L^2(\Omega)}\no\\
  \ls&\delta_r\norm{\Pi \nb^r \p}_{L^2(\D\Omega)}+C(1/\delta_r,K,K_1,M,\vol\Omega)\sum_{\ell=1}^{r-1} E_\ell(t)\no\\
  &+(r-2)C(1/\delta_r,K,K_1,M,\vol\Omega)E_2^{1/2}(t)E_r^{1/2}(t).
\end{align}
Now we estimate the boundary terms. Since $\p=\frac{1}{8\pi}\bv^2$ on $\D\Omega$, by \eqref{eq.CL00prop5.9.1}, we have for $r\gs 1$
\begin{align}\label{eq.pinbrp}
    \norm{\Pi\nb^r \p}_{L^2(\D\Omega)}\ls & C(K,K_1)\left(\norm{\theta}_{L^\infty(\D\Omega)}+(r-2)\sum_{k\ls r-3} \norm{\bnb^k \theta}_{L^2(\D\Omega)}\right)\no \\
    &\times\sum_{k\ls r-1}\norm{\nb^k \p}_{L^2(\D\Omega)}.
  \end{align}

From \eqref{eq.CL004.20}, we get $\Pi\nb^2\p=\theta \nb_N\p$ and then, by \eqref{eq.2energy91}, \eqref{eq.2energy9}, \eqref{eq.CL00lemA.7.1}, \eqref{eq.2energy8} and \eqref{eq.est.nbrp}, we get
\begin{align}
  \norm{\theta}_{L^2(\D\Omega)}=&\norm{\frac{\Pi\nb^2\p}{\nb_N\p}}_{L^2(\D\Omega)}\ls \frac{1}{\eps} \norm{\Pi\nb^2\p}_{L^2(\D\Omega)},\\
  \norm{\Pi\nb^2\p}_{L^2(\D\Omega)}\ls& \norm{\theta}_{L^\infty(\D\Omega)}\norm{\nb\p}_{L^2(\D\Omega)}\no\\
  \ls&C(K,\vol\Omega)\left(\norm{\nb^2\p}_{L^2(\Omega)}+\norm{\nb\p}_{L^2(\Omega)}\right)\no\\
  \ls&C(K,\vol\Omega)\delta_2\norm{\Pi\nb^2\p}_{L^2(\D\Omega)} +C(K,\vol\Omega)(\vol\Omega)^{1/2} M\no\\
  &+C(1/\delta_2,K,K_1,M,\vol\Omega)E_1(t),\label{eq.pinb2p}
\end{align}
where the first term of the right hand side of \eqref{eq.pinb2p} can be absorbed by the left hand side if we take $\delta_2$ so small that, e.g., $C(K,\vol\Omega)\delta_2\ls 1/2$. Thus, it follows that
\begin{align}
  \norm{\Pi\nb^2\p}_{L^2(\D\Omega)}\ls&C(K,K_1,M,\vol\Omega)(1+E_1(t)),\\
  \norm{\nb^2\p}_{L^2(\Omega)}\ls&C(K,K_1,M,\vol\Omega)(1+E_1(t)),\label{eq.nb2p}\\
  \norm{\theta}_{L^2(\D\Omega)}\ls&C(K,K_1,M,\vol\Omega,1/\eps)(1+E_1(t)).\label{eq.theta}
\end{align}

By Theorem \ref{thm.1energy}, there exists a $T>0$ such that $E_1(t)$ can be controlled by the initial energy $E_1(0)$ for $t\in[0,T]$, e.g., $E_1(t)\ls 2E_1(0)$.
Thus, from \eqref{eq.pinbrp}, \eqref{eq.theta}, \eqref{eq.2energy8} and \eqref{eq.nb2p} we have
\begin{align}
  \norm{\Pi\nb^3\p}_{L^2(\D\Omega)}\ls &C(K,K_1)\left(K+\norm{\theta}_{L^2(\D\Omega)}\right) \sum_{k\ls 2}\norm{\nb^k\p}_{L^2(\D\Omega)}\no\\
  \ls&C(K,K_1,M,\vol\Omega,1/\eps)(1+E_1(t))\sum_{k\ls 3}\norm{\nb^k\p}_{L^2(\Omega)}\no\\
  \ls& C(K,K_1,M,\vol\Omega,1/\eps,E_1(0))\norm{\nb^3\p}_{L^2(\Omega)} \no\\ &+C(K,K_1,M,\vol\Omega,1/\eps,E_1(0)).
\end{align}

From \eqref{eq.est.nbrp},
\begin{align}
  \norm{\nb^3 \p}_{L^2(\Omega)}
  \ls& \delta_3 C(K,K_1,M,\vol\Omega,1/\eps,E_1(0))\norm{\nb^3\p}_{L^2(\Omega)} \no\\
  &+\delta_3 C(K,K_1,M,\vol\Omega,1/\eps,E_1(0))\no\\
  &+C(1/\delta_3,K,K_1,M,\vol\Omega)\sum_{\ell=1}^2 E_\ell(t)\no\\
  &+C(1/\delta_3,K,K_1,M,\vol\Omega)E_2^{1/2}(t)E_3^{1/2}(t),
\end{align}
which, if we choose $\delta_3>0$ so small that $$\delta_3 C(K,K_1,M,\vol\Omega,1/\eps,E_1(0))\ls 1/2,$$
yields
\begin{align}\label{eq.nb3p}
  \norm{\nb^3 \p}_{L^2(\Omega)}
  \ls&  C(K,K_1,M,\vol\Omega,1/\eps,E_1(0))+C(K,K_1,M,\vol\Omega)\sum_{\ell=1}^2 E_\ell(t)\no\\
  &+C(K,K_1,M,\vol\Omega)E_2^{1/2}(t)E_3^{1/2}(t),
\end{align}
and then
\begin{align}
  \norm{\Pi\nb^3 \p}_{L^2(\D\Omega)}
  \ls&  C(K,K_1,M,\vol\Omega,1/\eps,E_1(0))\no\\
  &\times\left(1+\sum_{\ell=1}^2 E_\ell(t)+E_2^{1/2}(t)E_3^{1/2}(t)\right).
\end{align}

Since
\begin{align*}
  \bnb_b\nb_N \p=&\gamma_b^d\nb_d (N^a\nb_a \p)
  =(\delta_b^d-N_bN^d)((\nb_dN^a)\nb_a \p+N^a\nb_d\nb_a\p)\\
  =&\theta_b^a\nb_a \p+N^a\nb_b\nb_a\p-N_bN^d(\theta_d^a\nb_a\p+N^a\nb_d\nb_a\p),
\end{align*}
from \eqref{eq.CL00lemA.7.1},   it follows that
\begin{align*}
  \norm{\bnb\nb_N\p}_{L^2(\D\Omega)}
  \ls& C\norm{\theta}_{L^\infty(\D\Omega)}\norm{\nb \p}_{L^2(\D\Omega)}+C\norm{\nb^2\p}_{L^2(\D\Omega)}\\
  \ls&C(K,\vol\Omega)\left(\norm{\nb^3 \p}_{L^2(\Omega)}+\norm{\nb^2 \p}_{L^2(\Omega)}+\norm{\nb \p}_{L^2(\Omega)}\right)\\
  \ls& C(K,K_1,M,\vol\Omega,1/\eps,E_1(0))+C(K,K_1,M,\vol\Omega)\sum_{\ell=1}^2 E_\ell(t)\\
  &+C(K,K_1,M,\vol\Omega)E_2^{1/2}(t)E_3^{1/2}(t).
\end{align*}
Thus, by \eqref{eq.CL004.21}, it follows that $(\bnb\theta)\nb_N \p=\Pi\nb^3 \p-3\theta\tilde{\otimes}\bnb\nb_N \p$ and
\begin{align}\label{eq.theta1}
  \norm{\bnb\theta}_{L^2(\D\Omega)}\ls& \frac{1}{\eps}\left(\norm{\Pi\nb^3 \p}_{L^2(\D\Omega)}+C\norm{\theta}_{L^\infty(\D\Omega)}\norm{\bnb\nb_N \p}_{L^2(\D\Omega)}\right)\no\\
  \ls&C(K,K_1,M,\vol\Omega,1/\eps,E_1(0))\no\\
  &\times\left(1+\sum_{\ell=1}^2 E_\ell(t)+E_2^{1/2}(t)E_3^{1/2}(t)\right).
\end{align}

Hence, from \eqref{eq.pinbrp}, \eqref{eq.CL00lemA.7.1}, it yields
\begin{align}
  \norm{\Pi\nb^4\p}_{L^2(\D\Omega)}\ls &C(K,K_1)\left(K+\norm{\theta}_{L^2(\D\Omega)}+ \norm{\bnb \theta}_{L^2(\D\Omega)}\right)\no\\
  &\qquad\qquad\times\sum_{k\ls 4}\norm{\nb^k\p}_{L^2(\Omega)}.\label{eq.nb4p}
\end{align}
Then, from \eqref{eq.est.nbrp}, we can absorb the highest order term $\norm{\nb^4\p}_{L^2(\Omega)}$ by the left hand side for $\delta_4>0$ small enough which is independent of the highest energy $E_4(t)$, and get
\begin{align}
  \norm{\nb^4\p}_{L^2(\Omega)}\ls& C(K,K_1,M,\vol\Omega,1/\eps,E_1(0))\no\\
  &\qquad\times\left(1+\sum_{\ell=1}^3 E_\ell(t)+E_2^{1/2}(t)E_4^{1/2}(t)\right),\\
  \norm{\Pi\nb^4\p}_{L^2(\D\Omega)}\ls& C(K,K_1,M,\vol\Omega,1/\eps,E_1(0))\no\\
  &\qquad\times\left(1+\sum_{\ell=1}^3 E_\ell(t)+E_2^{1/2}(t)E_4^{1/2}(t)\right).
\end{align}

Therefore, from \eqref{eq.nb2p}, \eqref{eq.nb3p} and \eqref{eq.nb4p}, we obtain for $r\gs 2$
\begin{align}
  \norm{\nb^r\p}_{L^2(\Omega)}\ls &C(K,K_1,M,\vol\Omega,1/\eps,E_1(0))\no\\
  &\qquad\times\left(1+\sum_{\ell=1}^{r-1} E_\ell(t)+(r-2)E_2^{1/2}(t)E_r^{1/2}(t)\right),
\end{align}
which, from \eqref{eq.up}, implies
\begin{align}
  \eqref{eq.r.e12}\ls& C(K,K_1,M,\vol\Omega,1/\eps,E_1(0))E_r^{1/2}(t)\no\\
  &\qquad\times\left(1+\sum_{\ell=1}^{r-1} E_\ell(t)+(r-2)E_2^{1/2}(t)E_r^{1/2}(t)\right).
\end{align}

Now, we turn to the estimates of  \eqref{eq.r.e11}. Since $\p=\frac{1}{8\pi}\bv^2$ on $\D\Omega$ implies $\gamma_b^a\nb_a \p=0$ on $\D\Omega$, we get from \eqref{gammaauc}, by noticing that $\vartheta=-1/\nb_N \p$, that
\begin{align}\label{eq.nunb}
  -\vartheta^{-1}N_b=&\nb_N \p N_b=N^a\nb_a \p N_b=\delta_b^a\nb_a\p-\gamma_b^a\nb_a\p
  =\nb_b \p.
\end{align}
 By the H\"older inequality and \eqref{eq.nunb}, we have
\begin{align}
  \eqref{eq.r.e11}\ls &C\norm{\vartheta}_{L^\infty(\D\Omega)}^{1/2} E_r^{1/2}(t)\norm{\Pi\left(D_t\left(\nb^{r} \p\right)-\vartheta^{-1}N_b\nb^{r} u^b\right)}_{L^2(\D\Omega)}\no\\
  = &C\norm{\vartheta}_{L^\infty(\D\Omega)}^{1/2} E_r^{1/2}(t)\norm{\Pi\left(D_t\left(\nb^{r} \p\right)+\nb^{r} u\cdot\nb \p\right)}_{L^2(\D\Omega)}.
\end{align}
By \eqref{eq.commutator}, it follows that
\begin{align}\label{eq.Piterm1}
  D_t\nb^r \p+&\nb^{r} u\cdot\nb \p=[D_t,\nb^r]\p+\nb^r D_t \p+\nb^{r} u\cdot\nb \p\no\\
  =&\sgn(2-r)\sum_{s=1}^{r-2}\left(r\atop s+1\right)(\nb^{s+1}u)\cdot \nb^{r-s} \p+\nb^r D_t \p.
\end{align}

We first consider the estimates of the last term in \eqref{eq.Piterm1}. By \eqref{eq.CL00prop5.9.1} and \eqref{eq.CL00lemA.7.1}, we get, for $2\ls r\ls 4$ 
\begin{align}\label{eq.est.pinbdtp}
  &\norm{\Pi\nb^rD_t\p}_{L^2(\D\Omega)}\no\\
  \ls& C(K,K_1,\vol\Omega)\left(\norm{\theta}_{L^\infty(\D\Omega)}+(r-2)\sum_{k\ls r-3}\norm{\bnb^k\theta}_{L^2(\D\Omega)}\right)\no\\
  &\times\sum_{k\ls r} \norm{\nb^k D_t\p}_{L^2(\Omega)}.
\end{align}

From \eqref{eq.CL00prop5.8.2} , it follows that
\begin{align}\label{eq.est.nbkdtp}
  &\norm{\nb^r D_t\p}_{L^2(\Omega)}\no\\
  \ls& \delta \norm{\Pi\nb^r D_t\p}_{L^2(\D\Omega)}+C(1/\delta, K,\vol\Omega)\sum_{s\ls r-2}\norm{\nb^s\Delta D_t \p}_{L^2(\Omega)}.
\end{align}

By \eqref{eq.CL002.4.3}, \eqref{eq.2e.p}, Lemma \ref{lem.CL00lem2.1}, \eqref{eq.1energy1}, \eqref{eq.1energy2} and \eqref{mhd4},  it yields
\begin{align*}
  \Delta D_t \p=&2h^{ab}\nb_a\nb_b\p+(\Delta u^e)\nb_e \p-D_t(g^{bd}g^{ac}\nb_a u_d\nb_b u_c)\\
  &+\frac{1}{4\pi}D_t(g^{bd}g^{ac}\nb_a \beta_d \nb_b \beta_c)\\
  =&2h^{ab}\nb_a\nb_b\p+(\Delta u^e)\nb_e \p-2D_t(g^{bd})\nb_a u_d\nb_b u^a\\
  &-2 g^{bd}D_t(\nb_a u_d)\nb_b u^a+\frac{1}{2\pi} D_t(g^{bd}) g^{ac}\nb_a \beta_d \nb_b \beta^a \\
  &+ \frac{1}{2\pi}g^{bd} D_t(\nb_a \beta_d)\nb_b\beta^a\\
  =&2h^{ab}\nb_a\nb_b\p+(\Delta u^e)\nb_e \p+4h^{bd}\nb_a u_d\nb_b u^a-\frac{1}{\pi} h^{bd} \nb_a \beta_d \nb_b\beta^a\\
  &+2 g^{bd}\nb_b u^a \nb_a\nb_d\p -2 g^{bd}\nb_b u^a\nb_a u^c\nb_du_c\\
  &-\frac{1}{2\pi} \nb_b u^a(\nb_a\beta^c \nb_c\beta^b+\beta^c\nb_a\nb_c\beta^b)\\
  &+ \frac{1}{2\pi}g^{bd}\nb_b\beta^a\left(\nb_a\beta^e(\nb_e u_d+\nb_d u_e)+\beta^e\nb_e\nb_a u_d\right)\\
  =&4g^{ac}\nb_c u^b\nb_a\nb_b\p+(\Delta u^e)\nb_e \p+2\nb_e u^b\nb_b u^a\nb_a u^e\\
  &-\frac{1}{2\pi} \nb_b u^a\nb_a\beta^c \nb_c\beta^b-\frac{1}{2\pi} \nb_b u^a\beta^c\nb_a\nb_c\beta^b+ \frac{1}{2\pi}\nb_b\beta^a\beta^e\nb_e\nb_a u^b.
\end{align*}

By \eqref{eq.r.einfbeta}, \eqref{eq.est.nbrp} and Lemma \ref{lem.CL00lemA.4}, it follows that for $s\ls 2$
\begin{align}\label{eq.est.nbdtp}
  &\norm{\nb^s\Delta D_t\p}_{L^2(\Omega)}\no\\
  \ls&C\norm{\nb  u}_{L^\infty(\Omega)}\norm{\nb^{s+2} \p}_{L^2(\Omega)}+s(s-1)C\norm{\nb^3 u}_{L^2(\Omega)}\norm{\nb^2 \p}_{L^\infty(\Omega)}\no\\
  &+sC\norm{ \nb^2 u}_{L^4(\Omega)}\norm{\nb^{s+1} \p}_{L^4(\Omega)}+C\norm{\nb^{s+2} u}_{L^2(\Omega)}\norm{\nb  \p}_{L^\infty(\Omega)}\no\\
  &+C\left(\norm{\nb u}_{L^\infty(\Omega)}\norm{\nb u}_{L^\infty(\Omega)}+\norm{\nb \beta}_{L^\infty(\Omega)}\norm{\nb \beta}_{L^\infty(\Omega)}\right) \norm{\nb^{s+1} u}_{L^2(\Omega)}\no\\
  &+s(s-1)C\norm{\nb u}_{L^\infty(\Omega)}\norm{\nb^2 u}_{L^4(\Omega)} \norm{\nb^2 u}_{L^4(\Omega)}\no\\
  &+C\norm{\nb u}_{L^\infty(\Omega)}\norm{\nb \beta}_{L^\infty(\Omega)}\norm{\nb^{s+1} \beta}_{L^2(\Omega)}\no\\
  &+sC\norm{\nb^2 u}_{L^4(\Omega)}\norm{\nb^2 \beta}_{L^4(\Omega)} \left((s-1)\norm{\nb  \beta}_{L^\infty(\Omega)}+\norm{\beta}_{L^\infty(\Omega)}\right)\no\\
  &+s(s-1)C\norm{\nb u}_{L^\infty(\Omega)}\norm{\nb^2 \beta}_{L^4(\Omega)} \norm{\nb^2  \beta}_{L^4(\Omega)}\no\\
  &+C\norm{\nb u}_{L^\infty(\Omega)}\norm{\beta}_{L^\infty(\Omega)} \norm{\nb^{s+2} \beta}_{L^2(\Omega)}\no\\
  &+sC\norm{\nb^3 u}_{L^2(\Omega)}\norm{\beta}_{L^\infty(\Omega)} \left((s-1)\norm{\nb^{2} \beta}_{L^\infty(\Omega)}+\norm{\nb \beta}_{L^\infty(\Omega)}\right)\no\\
  &+s(s-1)C\norm{\nb^3 \beta}_{L^2(\Omega)}\norm{\beta}_{L^\infty(\Omega)} \norm{\nb^{2} u}_{L^\infty(\Omega)}\no\\
  &+s(s-1)C\norm{\nb \beta}_{L^\infty(\Omega)}\norm{\nb^2\beta}_{L^4(\Omega)} \norm{\nb^{2} u}_{L^4(\Omega)}\no\\
  &+s(s-1)C\norm{\nb \beta}_{L^\infty(\Omega)}\norm{\beta}_{L^\infty(\Omega)} \norm{\nb^4 u}_{L^2(\Omega)}\no\\
  &+s(s-1)C\norm{\nb^2 \beta}_{L^\infty(\Omega)}\norm{\beta}_{L^\infty(\Omega)} \norm{\nb^3 u}_{L^2(\Omega)}.
\end{align}
From \eqref{lem.CL00lemA.3}, \eqref{eq.r.einfu}, it follows that
\begin{align}
  \norm{\nb^{s+1} u}_{L^4(\Omega)}\ls& C\norm{\nb^s u}_{L^\infty(\Omega)}^{1/2}\left(\sum_{\ell=0}^2\norm{\nb^{s+\ell} u}_{L^2(\Omega)}K_1^{2-\ell}\right)^{1/2}\no\\
  \ls& C(K_1)\sum_{\ell=0}^2 E_{s+\ell}^{1/2}(t).
\end{align}
We can estimate all the terms with $L^4(\Omega)$ norms in the same way with the help of \eqref{eq.r.einfbeta}, \eqref{eq.r.einfu}, the similar estimate of $\p$ and the assumptions. Thus, we obtain the bound which is linear about the highest-order derivative or the highest-order energy $E_r^{1/2}(t)$, i.e.,
\begin{align}\label{eq.nbdtp}
  \norm{\nb^s\Delta D_t\p}_{L^2(\Omega)}\ls& C(K,K_1,M,M_1,L,1/\eps,\vol\Omega,E_0(0))\no\\
   &\qquad\times\Big(1+\sum_{\ell=0}^{r-1}E_\ell(t)\Big)\big(1+E_r^{1/2}(t)\big).
\end{align}
Thus, from \eqref{eq.est.pinbdtp}, \eqref{eq.est.nbkdtp}, \eqref{eq.nbdtp} and taking some small $\delta$'s which are independent of $E_r(t)$, we obtain, by induction argument for $r$, that
\begin{align}\label{eq.pinbrdtp}
  \norm{\Pi \nb^r D_t\p}_{L^2(\D\Omega)}\ls &C(K,K_1,M,M_1,L,1/\eps,\vol\Omega,E_0(0))\no\\
   &\qquad\times\Big(1+\sum_{\ell=0}^{r-1}E_\ell(t)\Big)\big(1+E_r^{1/2}(t)\big).
\end{align}

To estimate \eqref{eq.Piterm1}, it only remains to estimate
\begin{align}
    \norm{\Pi\left((\nb^{s+1}u)\cdot \nb^{r-s} \p\right)}_{L^2(\D\Omega)} \text{ for } 1\ls s\ls r-2.
\end{align}
For $r=3,4$ and $s=r-2$, we have, by \eqref{eq.2energy92} and Lemma \ref{lem.CL00lemA.7}, that
\begin{align}
  &\norm{\Pi\left((\nb^{r-1}u)\cdot \nb^2 \p\right)}_{L^2(\D\Omega)}\no\\
  \ls& \norm{\nb^{r-1} u}_{L^2(\D\Omega)}\norm{\nb^2 \p}_{L^\infty(\D\Omega)}\no\\
  \ls& CL\norm{\nb^2 u}_{L^{2(n-1)/(n-2)}(\D\Omega)}\no\\
  \ls &C(K,\vol\Omega)L\left(\norm{\nb^r u}_{L^2(\Omega)}+\norm{\nb^{r-1} u}_{L^2(\Omega)}\right)\no\\
  \ls &C(K,L,\vol\Omega)\left(E_{r-1}^{1/2}(t)+E_r^{1/2}(t)\right).
\end{align}
For $n=3$, $r=4$ and $s=1$, by \eqref{eq.CL004.48}, Lemma \ref{lem.CL00lemA.7} and \eqref{eq.est.nbrp}, we get
\begin{align}\label{eq.piup.1}
  &\norm{\Pi\left((\nb^2u)\cdot \nb^3 \p\right)}_{L^2(\D\Omega)}\no\\
  =&\norm{\Pi\nb^2u\cdot \Pi\nb^3 \p+\Pi(\nb^2 u\cdot N)\tilde{\otimes}\Pi(N\cdot \nb^3\p)}_{L^2(\D\Omega)}\no\\
  \ls&C\norm{\Pi\nb^2 u}_{L^4(\D\Omega)}\norm{\Pi\nb^3\p}_{L^4(\D\Omega)}\no\\
  & +C\norm{\Pi(N^a\nb^2 u_a)}_{L^4(\D\Omega)}\norm{\Pi(\nb_N\nb^2 \p)}_{L^4(\D\Omega)}\no\\
  \ls &C\norm{\nb^2 u}_{L^4(\D\Omega)}\norm{\nb^3 \p}_{L^4(\D\Omega)}\no\\
  \ls&C(K,\vol\Omega)\left(\norm{\nb^3 u}_{L^2(\Omega)}+\norm{\nb^2 u}_{L^2(\Omega)}\right) \left(\norm{\nb^4 \p}_{L^2(\Omega)}+\norm{\nb^3 \p}_{L^2(\Omega)}\right)\no\\
  \ls&C(K, K_1,\vol\Omega)(E_3^{1/2}(t)+E_2^{1/2}(t))\left(\sum_{s=0}^3 E_s(t)+\left(\sum_{\ell=0}^{2}E_\ell^{1/2}(t)\right)E_4^{1/2}(t)\right)\no\\
  \ls &C(K, K_1,\vol\Omega)\sum_{s=0}^3 E_s(t)\sum_{\ell=0}^4 E_\ell^{1/2}(t).
\end{align}

Hence, we have
\begin{align}
  \eqref{eq.r.e11}\ls& C(K,K_1,M,M_1,L,1/\eps,\vol\Omega,E_0(0))\no\\
   &\qquad\times\Big(1+\sum_{s=0}^{r-1}E_s(t)\Big)\big(1+E_r(t)\big).
\end{align}
By Lemma \ref{lem.CL00lemA.3}, we can obtain
\begin{align}
  \eqref{eq.r.e111}+ \eqref{eq.r.e112}\ls C(K,K_1,M,\vol\Omega,1/\eps)\left(1+\sum_{s=0}^{r-1} E_s(t)\right)E_r(t).
\end{align}
Therefore, we have shown that
\begin{align}
  \eqref{eq.r.e1}+\eqref{eq.r.e2}+\eqref{eq.r.e7}\ls &C(K,K_1,M,M_1,L,1/\eps,\vol\Omega,E_0(0))\no\\
   &\qquad\times\Big(1+\sum_{s=0}^{r-1}E_s(t)\Big)\big(1+E_r(t)\big).
\end{align}

 We now calculate the material derivatives of $|\nb^{r-1}\curl u|^2$ and $|\nb^{r-1}\curl \beta|^2$. From Lemma \ref{lem.CL00lem2.1}, \eqref{eq.r.u} and \eqref{eq.nb.beta}, we have
\begin{align*}
  &D_t\left(|\nb^{r-1}\curl u|^2+\frac{1}{4\pi}|\nb^{r-1}\curl \beta|^2\right)\\
  =&D_t\left(g^{ac}g^{bd}g^{AF}\nb_A^{r-1}(\curl u)_{ab} \nb_F^{r-1}(\curl u)_{cd}\right)\\
  &+\frac{1}{4\pi}D_t\left(g^{ac}g^{bd}g^{AF}\nb_A^{r-1}(\curl \beta)_{ab} \nb_F^{r-1}(\curl \beta)_{cd}\right)\\
  =&(r+1)D_t(g^{ac})g^{bd}g^{AF}\nb_A^{r-1}(\curl u)_{ab} \nb_F^{r-1}(\curl u)_{cd}\\
   &+4 g^{ac}g^{bd}g^{AF}D_t\left(\nb_A^{r-1}\nb_a u_b\right) \nb_F^{r-1}(\curl u)_{cd}\\
   &+\frac{r+1}{4\pi}D_t(g^{ac})g^{bd}g^{AF}\nb_A^{r-1}(\curl \beta)_{ab} \nb_F^{r-1}(\curl \beta)_{cd}\\
   &+\frac{1}{\pi} g^{ac}g^{bd}g^{AF}D_t\left(\nb_A^{r-1}\nb_a \beta_b\right) \nb_F^{r-1}(\curl \beta)_{cd}\\
  =&-2(r+1)g^{ae}\nb_e u^cg^{bd}g^{AF}\nb_A^{r-1}(\curl u)_{ab} \nb_F^{r-1}(\curl u)_{cd}\\
  &-\frac{r+1}{2\pi}g^{ae}\nb_e u^cg^{bd}g^{AF}\nb_A^{r-1}(\curl \beta)_{ab} \nb_F^{r-1}(\curl \beta)_{cd}\\
  &-4 g^{ac}g^{bd}g^{AF}\nb_F^{r-1}(\curl u)_{cd}\nb_{Aa}^r\nb_b\p \quad\text{ (this vanishes by symmetry)}\\
  &+ 4g^{ac}g^{bd}g^{AF}\nb_F^{r-1}(\curl u)_{cd}(\curl u)_{be}\nb_{Aa}^r u^e\\
  &+4\sgn(2-r)g^{ac}g^{AF}\nb_F^{r-1}(\curl u)_{cd}\sum_{s=1}^{r-2}\left(r\atop s+1\right)\left((\nb^{1+s}u)\cdot\nb^{r-s}u^d\right)_{Aa}\\
  &+\frac{1}{\pi}\sgn(2-r)g^{ac}g^{AF}\nb_F^{r-1}(\curl \beta)_{cd}\sum_{s=1}^{r-2}\left(r\atop s+1\right)\left((\nb^{1+s}u)\cdot\nb^{r-s}\beta^d\right)_{Aa}\\
&+\frac{1}{\pi} g^{ac}g^{bd}g^{AF}\nb_F^{r-1}(\curl \beta)_{cd}\nb_{Aa}^r\beta^e\nb_b u_e- \frac{1}{\pi}g^{ac}g^{AF}\nb_F^{r-1}(\curl \beta)_{cd}\nb_{Aa}^r u^e\nb_e\beta^d\\
  &+\frac{1}{\pi}\nb_e \left(g^{ac}g^{AF}\beta^e\nb_F^{r-1}(\curl u)_{cd}\nb_{Aa}^r \beta^d\right)\\
  &+\frac{1}{\pi}g^{ac}g^{AF}\nb_F^{r-1}(\curl u)_{cd}\sum_{s=1}^r\left(r\atop s\right) \left(\nb^s\beta^e\nb^{r-s}\nb_e\beta^d\right)_{Aa}\\
  &+\frac{1}{\pi}g^{ac}g^{AF}\nb_F^{r-1}(\curl \beta)_{cd}\sum_{s=1}^r\left(r\atop s\right) \left(\nb^s\beta^e\nb^{r-s}\nb_e u^d\right)_{Aa}.
\end{align*}
Noticing that $\beta\cdot N=0$ on $\D\Omega$, then by the H\"older inequality and the Gauss formula, we get
\begin{align}
  \eqref{eq.r.e3}\ls C(K,K_1,M,\vol\Omega,1/\eps)\left(1+\sum_{s=0}^{r-1} E_s(t)\right)E_r(t).
\end{align}

  Thus, by \eqref{eq.CL00lem3.9.1} and \eqref{eq.Dtpcommu}, we get
  \begin{align*}
    D_t(\nb_N \p)=&D_t(N^a\nb_a \p)=(D_t N^a)\nb_a \p+N^aD_t\nb_a \p\\
    =&(-2h_d^a N^d+h_{NN} N^a)\nb_a \p+N^a\nb_a D_t\p\\
    =&-2h_d^a N^d\nb_a \p+h_{NN}\nb_N \p+\nb_N D_t\p,
  \end{align*}
  which yields
  \begin{align}
    \frac{\vartheta_t}{\vartheta}=-\frac{D_t\nb_N \p}{\nb_N\p}=\frac{2h_d^a N^d\nb_a \p}{\nb_N\p}-h_{NN}+\frac{\nb_ND_t\p}{\nb_N\p}.
  \end{align}
Thus, we can easily obtain that the remainder integrals, i.e., \eqref{eq.r.e4}, \eqref{eq.r.e5}, \eqref{eq.r.e6} and \eqref{eq.r.e8}, can be controlled by $C(K,M,L,1/\eps)E_r(t)$.

Therefore, we obtain
\begin{align}
  \frac{d}{dt}E_r(t)\ls &C(K,K_1,M,M_1,L,1/\eps,\vol\Omega,E_0(0))\no\\
   &\qquad\times\Big(1+\sum_{s=0}^{r-1}E_s(t)\Big)\big(1+E_r(t)\big),
\end{align}
which implies the desired result \eqref{eq.renergy} by Gronwall's inequality and the induction argument for $r\in\{2,\cdots, n+1\}$.
\end{proof}

\section{Justification of A Priori Assumptions}

Let  $\K(t)$ and $\eps(t)$ be the maximum and minimum values, respectively, such that \eqref{eq.2energy9} and \eqref{eq.2energy91} hold at time $t$:
\begin{align}
  \K(t)=&\max\left(\norm{\theta(t,\cdot)}_{L^\infty(\D\Omega)}, 1/\iota_0(t)\right),\label{eq.K}\\
  \E(t)=&\norm{1/(\nb_N \p(t,\cdot))}_{L^\infty(\D\Omega)}=1/\eps(t). \label{eq.E}
\end{align}

\begin{lemma}\label{lem.7.6}
  Let $K_1\gs 1/\iota_1$ be as in Definition \ref{defn.3.5}, $\E(t)$ as in \eqref{eq.E}. Then there are continuous functions $G_j$, $j=1,2,3,4$, such that
  \begin{align}
    \norm{\nb u}_{L^\infty(\Omega)}+\norm{\nb \beta}_{L^\infty(\Omega)}+&\norm{ \beta}_{L^\infty(\Omega)}\ls G_1(K_1,E_0,\cdots, E_{n+1}),\label{eq.e.1}\\
    \norm{\nb \p}_{L^\infty(\Omega)}+\norm{\nb^2\p}_{L^\infty(\D\Omega)}\ls &G_2(K_1,\E,E_0,\cdots, E_{n+1},\vol\Omega),\label{eq.e.2}\\
    \norm{\theta}_{L^\infty(\D\Omega)}\ls &G_3(K_1,\E,E_0,\cdots, E_{n+1},\vol\Omega),\label{eq.e.3}\\
    \norm{\nb D_t\p}_{L^\infty(\D\Omega)}\ls &G_4(K_1,\E,E_0,\cdots, E_{n+1},\vol\Omega).\label{eq.e.4}
  \end{align}
\end{lemma}

\begin{proof}
  \eqref{eq.e.1} follows from \eqref{eq.r.einfu}, \eqref{eq.r.einfbeta} and \eqref{eq.CllemmaA.4}. From Lemmas \ref{lem.CL00lemA.4} and \ref{lem.CL00lemA.2},  we have
\begin{align}
  \norm{\nb \p}_{L^\infty(\Omega)}\ls &C(K_1)\sum_{\ell=0}^{2} \norm{\nb^{\ell+1}\p}_{L^2(\Omega)},\label{eq.e.5}\\
  \norm{\nb^2\p}_{L^\infty(\D\Omega)}\ls &C(K_1)\sum_{\ell=0}^{n+1} \norm{\nb^{\ell}\p}_{L^2(\D\Omega)}.\label{eq.e.6}
\end{align}
Thus, \eqref{eq.e.2} follows from \eqref{eq.e.5}, \eqref{eq.e.6}, Lemmas \ref{lem.CL00lemA.5}--\ref{lem.CL00lemA.7}, \eqref{eq.deltap2}, \eqref{eq.nb2p} and \eqref{eq.nb3p}. Since, from \eqref{eq.CL004.20},
\begin{align}\label{eq.e.7}
  |\nb^2\p|\gs |\Pi\nb^2\p|=|\nb_N\p||\theta|\gs \E^{-1}|\theta|,
\end{align}
so \eqref{eq.e.3} follows from \eqref{eq.e.2}. \eqref{eq.e.4} follows from Lemma \ref{lem.CL00lemA.2}, \eqref{eq.est.nbkdtp}, \eqref{eq.nbdtp} and \eqref{eq.pinbrdtp}.
\end{proof}

\begin{lemma}\label{lem.7.7}
  Let $K_1\gs 1/\iota_1$ and $\eps_1$ be as in Definition \ref{defn.3.5}. Then
  \begin{align}\label{eq.e.8}
    \abs{\frac{d}{dt}E_r}\ls C_r(K_1,\E,E_0,\cdots, E_{n+1},\vol\Omega)\sum_{s=0}^r E_s,
  \end{align}
  and
  \begin{align}\label{eq.e.9}
    \abs{\frac{d}{dt}\E}\ls C_r(K_1,\E,E_0,\cdots, E_{n+1},\vol\Omega).
  \end{align}
\end{lemma}

\begin{proof}
  \eqref{eq.e.8} is a consequence of Lemma \ref{lem.7.6} and the estimates in the proof of Theorems \ref{thm.1energy} and \ref{thm.renergy}. \eqref{eq.e.9} follows from
  \begin{align*}
    \abs{\frac{d}{dt}\norm{\frac{1}{-\nb_N\p(t,\cdot)}}_{L^\infty(\D\Omega)}}\ls C\norm{\frac{1}{-\nb_N\p(t,\cdot)}}_{L^\infty(\D\Omega)}^2 \norm{\nb_ND_t\p(t,\cdot)}_{L^\infty(\D\Omega)}
  \end{align*}
  and \eqref{eq.e.4}.
\end{proof}

As a result of Lemma \ref{lem.7.7}, we have the following:

\begin{lemma}\label{lem.7.8}
  There exists a continuous function $\T>0$ depending on $K_1$, $\E(0)$, $E_0(0)$, $\cdots$, $E_{n+1}(0)$, $\vol\Omega$ such that for
  \begin{align}
    0\ls t\ls \T(K_1,\E(0),E_0(0),\cdots, E_{n+1}(0),\vol\Omega),
  \end{align}
  the following statements hold: We have
  \begin{align}\label{eq.7.37}
    E_s(t)\ls 2E_s(0), \quad 0\ls s\ls n+1, \quad \E(t)\ls 2\E(0).
  \end{align}
  Furthermore,
  \begin{align}\label{eq.7.38}
    \frac{g_{ab}(0,y)Y^aY^b}{2}\ls g_{ab}(t,y)Y^aY^b\ls 2g_{ab}(0,y)Y^aY^b,
  \end{align}
  and with $\eps_1$ as in Definition \ref{defn.3.5},
  \begin{align}
    \qquad\abs{\N(x(t,\bar{y}))-\N(x(0,\bar{y}))}\ls&\frac{\eps_1}{16}, &&\bar{y}\in\D\Omega,\qquad\label{eq.7.39}\\
    \abs{x(t,y)-x(t,y)}\ls&\frac{\iota_1}{16}, &&y\in\Omega,\label{eq.7.40}\\
    \abs{\frac{\D x(t,\bar{y})}{\D y}-\frac{\D (0,\bar{y})}{\D y}}\ls &\frac{\eps_1}{16}, &&\bar{y}\in\D\Omega.\label{eq.7.41}
  \end{align}
\end{lemma}

\begin{proof}
  We get \eqref{eq.7.37} from Lemma \ref{lem.7.7} if $\T(K_1,\E(0),E_0(0),\cdots, E_{n+1}(0)$, $\vol\Omega)>0$ is sufficiently small. Then from \eqref{eq.7.37} and Lemma \ref{lem.7.6}, we have
  \begin{align}
    \norm{\nb u}_{L^\infty(\Omega)}+&\norm{\nb \beta}_{L^\infty(\Omega)} +\norm{ \beta}_{L^\infty(\Omega)}+\norm{\nb \p}_{L^\infty(\Omega)}\no\\
    &\ls C(K_1,\E(0),E_0(0),\cdots, E_{n+1}(0)),\label{eq.7.42}\\
    \norm{\nb^2\p}_{L^\infty(\D\Omega)}+&\norm{\theta}_{L^\infty(\D\Omega)}\no\\
    &\ls C(K_1,\E(0),E_0(0),\cdots, E_{n+1}(0),\vol\Omega),\label{eq.7.43}\\
\norm{\nb D_t\p}_{L^\infty(\D\Omega)}&\ls C(K_1,\E(0),E_0(0),\cdots, E_{n+1}(0),\vol\Omega).\label{eq.7.44}
  \end{align}
  By \eqref{eq.1energy1} and \eqref{eq.1energy2}, we have
  \begin{align}
    \abs{D_t\nb u}\ls& \abs{\nb^2 \p}+\abs{\nb u}^2+\abs{\nb \beta}^2+\abs{\beta}\abs{\nb^2\beta},\\
    \abs{D_t \nb\beta}\ls& \abs{\nb\beta}\abs{\nb u}+\abs{\beta}\abs{\nb^2 u}.\label{eq.e.10}
  \end{align}
  By \eqref{eq.A.32}, \eqref{eq.CL00lemA.7.1}, Lemma \ref{lem.7.6} and \eqref{eq.7.37}, we have
  \begin{align*}
    \norm{\nb u}_{L^\infty(\D\Omega)}+\norm{\nb \beta}_{L^\infty(\D\Omega)}\ls C(K_1,\E(0),E_0(0),\cdots, E_{n+1}(0),\vol\Omega).
  \end{align*}
  Thus, by noticing that $|\beta|=\bv$ on $\D\Omega$, it follows, from \eqref{eq.7.42}, \eqref{eq.7.43}, Lemmas \ref{lem.CL00lemA.2} and \ref{lem.CL00lemA.7}, \eqref{eq.r.einfu} and \eqref{eq.r.einfbeta}, that
  \begin{align*}
    &\norm{D_t\nb u}_{L^\infty(\D\Omega)}+\norm{D_t\nb\beta}_{L^\infty(\D\Omega)}\\
    \ls& \norm{\nb^2\p}_{L^\infty(\D\Omega)}+\left(\norm{\nb u}_{L^\infty(\D\Omega)}+\norm{\nb \beta}_{L^\infty(\D\Omega)}\right)^2\no\\
    &+\bv\left(\norm{\nb^2 u}_{L^\infty(\D\Omega)}+\norm{\nb^2 \beta}_{L^\infty(\D\Omega)}\right)\no\\
    \ls& C(K_1,\E(0),E_0(0),\cdots, E_{n+1}(0),\vol\Omega)\\
    &\times\left(1+\norm{\nb u}_{L^\infty(\D\Omega)}+\norm{\nb \beta}_{L^\infty(\D\Omega)}\right),
  \end{align*}
  which yields, with the help of Gronwall's inequality, for $0\ls t\ls T$
  \begin{align}
    &\norm{\nb u(t,\cdot)}_{L^\infty(\D\Omega)}+\norm{\nb \beta(t,\cdot)}_{L^\infty(\D\Omega)}\no\\
\ls &e^{C(K_1,\E(0),E_0(0),\cdots, E_{n+1}(0),\vol\Omega)t}\left(\norm{\nb u(0,\cdot)}_{L^\infty(\D\Omega)}+\norm{\nb \beta(0,\cdot)}_{L^\infty(\D\Omega)}\right)\no\\
    &+e^{C(K_1,\E(0),E_0(0),\cdots, E_{n+1}(0),\vol\Omega)t}-1.
  \end{align}
  If $T$ is sufficiently small, it follows, after possibly making $\T>0$ smaller, that
  \begin{align}\label{eq.7.45}
    &\norm{\nb u(T,\cdot)}_{L^\infty(\D\Omega)}+\norm{\nb \beta(T,\cdot)}_{L^\infty(\D\Omega)}\no\\
\ls& 2\left(\norm{\nb u(0,\cdot)}_{L^\infty(\D\Omega)}+\norm{\nb \beta(0,\cdot)}_{L^\infty(\D\Omega)}\right),
  \end{align}
  which also guarantee the a priori assumption of \eqref{eq.beta0}.

By \eqref{eq.Dtpcommu}, \eqref{eq.A.4.2}, \eqref{eq.est.nbkdtp}, \eqref{eq.nbdtp} and \eqref{eq.pinbrdtp}, we have
 \begin{align*}
   \norm{D_t\nb \p}_{L^\infty(\Omega)}=&\norm{\nb D_t\p}_{L^\infty(\Omega)}\ls C(K_1)\sum_{\ell=0}^2\norm{\nb^{\ell+1} D_t\p}_{L^2(\Omega)}\\
   \ls & C(K_1,\E(0),E_0(0),\cdots, E_{n+1}(0),\vol\Omega),
 \end{align*}
 which implies for sufficiently small $T>0$
 \begin{align}
   \norm{\nb\p(t,\cdot)}_{L^\infty(\Omega)}\ls 2\norm{\nb\p(0,\cdot)}_{L^\infty(\Omega)}.
 \end{align}
 By \eqref{mhd3} and \eqref{eq.7.42}, we have
 \begin{align}
   \norm{D_t\vv}_{L^\infty(\dm)}\ls& \norm{\D \p}_{L^\infty(\dm)}+\norm{\B}_{L^\infty(\dm)}\norm{\D\B}_{L^\infty(\dm)}\\
   \ls&\norm{\nb \p}_{L^\infty(\Omega)}+\norm{\beta}_{L^\infty(\Omega)}\norm{\nb\beta}_{L^\infty(\Omega)}\\
   \ls&C(K_1,\E(0),E_0(0),\cdots, E_{n+1}(0)),
 \end{align}
 which yields
 \begin{align}\label{eq.7.47}
   \norm{\vv(t,\cdot)}_{L^\infty(\dm)}\ls 2\norm{\vv(0,\cdot)}_{L^\infty(\Omega)}.
 \end{align}

 \eqref{eq.7.38} follows from the same argument since $D_t g_{ab}=\nb_a u_b+\nb_b u_a$ and by \eqref{eq.7.42}
 \begin{align}
   &\abs{g_{ab}(T,y)Y^aY^b-g_{ab}(0,y)Y^aY^b}
   \ls \int_0^T\abs{D_t g_{ab}(s,y)} ds Y^aY^b\\
   \ls&2\int_0^T\norm{\nb_a u_b(s)}_{L^\infty(\Omega)} ds Y^aY^b \ls \frac{1}{2}g_{ab}(0,y)Y^aY^b,
 \end{align}
 if $T$ is sufficiently small. Now the estimate for $\N$ follows from
 $$D_t n_a=h_{NN} n_a,$$
 and the estimates for $x$ and $\D x/\D y$ from
 \begin{align}
   D_t x(t,y)=&v(t,x(t,y)),\\
   D_t\frac{\D x}{\D y} =&\frac{\D \vv(t,x(t,y))}{\D y} =\frac{\D \vv(t,x)}{\D x}\frac{\D x}{\D y},
 \end{align}
 and \eqref{eq.7.47} and \eqref{eq.7.45}, respectively.
\end{proof}

Now we use \eqref{eq.7.38}--\eqref{eq.7.41} to pick a $K_1$, i.e., $\iota_1$, which depends only on its value at $t=0$,
\begin{align}
  \iota_1(t)\gs \iota_1(0)/2.
\end{align}

\begin{lemma}\label{lem.7.9}
  Let $\T$ be as in Lemma \ref{lem.7.7}. Pick $\iota_1>0$ such that
  \begin{align}\label{eq.7.55}
    \abs{\N(x(0,y_1))-\N(x(0,y_2))}\ls \frac{\eps_1}{2}, \text{ whenever } \abs{x(0,y_1)-x(0,y_2)}\ls 2\iota_1.
  \end{align}
  Then if $t\ls \T$, we have
  \begin{align}\label{eq.7.56}
    \abs{\N(x(t,y_1))-\N(x(t,y_2))}\ls \eps_1, \text{ whenever } \abs{x(t,y_1)-x(t,y_2)}\ls 2\iota_1.
  \end{align}
\end{lemma}

\begin{proof}
  \eqref{eq.7.56} follows from \eqref{eq.7.55}, \eqref{eq.7.39} and \eqref{eq.7.40} in view of triangle inequalities.
\end{proof}

Lemma \ref{lem.7.9} allows us to pick a $K_1$ depending only on initial conditions, while Lemma \ref{lem.7.8} gives us $\T>0$ that depends only on the initial conditions and $K_1$ such that, by Lemma \ref{lem.7.9}, $1/\iota_1\ls K_1$ for $t\ls \T$. Thus, we immediately obtain the following theorem.

\begin{theorem}
  There exists a continuous function $\T>0$ such that if
  \begin{align}
    T\ls \T(\K(0),\E(0),E_0(0),\cdots, E_{n+1}(0),\vol\Omega),
  \end{align}
  any smooth solution of the free boundary problem for Euler's equations \eqref{mhd1} and \eqref{eq.phycond} for $0\ls t\ls T$ satisfies
  \begin{align}
    \sum_{s=0}^{n+1} E_s(t)\ls 2\sum_{s=0}^{n+1} E_s(0), \quad 0\ls t\ls T.
  \end{align}
\end{theorem}

\appendix

\section{Preliminaries and Some Estimates}

Let $N^a$ denote the unit normal to $\D\Omega$, $g_{ab}N^aN^b=1$, $g_{ab}N^a T^b=0$ if $T\in T(\D\Omega)$, and let $N_a=g_{ab}N^b$ denote the unit conormal, $g^{ab} N_aN_b=1$. The induced metric $\gamma$ on the tangent space to the boundary $T(\D\Omega)$ extended to be $0$ on the orthogonal complement in $T(\Omega)$ is then given by
\begin{align}
  \gamma_{ab}=g_{ab}-N_aN_b,\quad \gamma^{ab}=g^{ab}-N^aN^b.
\end{align}
The orthogonal projection of an $(r,s)$ tensor $S$ to the boundary is given by
\begin{align}
  (\Pi S)_{b_1\cdots b_s}^{a_1\cdots a_r}=\gamma_{c_1}^{a_1}\cdots \gamma_{c_r}^{a_r}\gamma_{b_1}^{d_1}\cdots \gamma_{b_s}^{d_s} S_{d_1\cdots d_s}^{c_1\cdots c_r},
\end{align}
where
\begin{align}\label{gammaauc}
  \gamma_a^c=\delta_a^c-N_aN^c.
\end{align}
Covariant differentiation on the boundary $\bnb$ is given by
\begin{align}
  \bnb S=\Pi\nb S.
\end{align}
The second fundamental form of the boundary is given by
\begin{align}\label{2ndfundform}
  \theta_{ab}=(\Pi\nb N)_{ab}=\gamma_a^c \nb_c N_b.
\end{align}

Let us now recall some properties of the projection. Since $g^{ab}=\gamma^{ab}+N^aN^b$, we have
\begin{align}\label{eq.CL004.48}
  \Pi(S\cdot R)=\Pi(S)\cdot \Pi(R)+\Pi(S\cdot N)\tilde{\otimes}\Pi(N\cdot R),
\end{align}
where $S\tilde{\otimes} R$ denotes some partial symmetrization of the tensor product $S\otimes R$, i.e., a sum over some subset of the permutations of the indices divided by the number of permutations in that subset. Similarly, we let $S\tilde{\cdot} R$ denote a partial symmetrization of the dot product $S\cdot R$. Now we recall some identities:
\begin{align}
  \Pi\nb^2 q=&\bnb^2 q+\theta \nb_N q,\label{eq.CL004.20}\\
  \Pi\nb^3 q=&\bnb^3 q-2\theta\tilde{\otimes}(\theta\tilde{\cdot}\bnb q)+(\bnb\theta)\nb_N q+3\theta\tilde{\otimes}\bnb\nb_N q,\label{eq.CL004.21}\\
  \Pi\nb^4 q=&\bnb^4 q-\theta\tilde{\otimes}\left(5(\bnb\theta)\tilde{\cdot}\bnb q+8\theta\tilde{\cdot}\bnb^2 q\right)-2(\bnb\theta)\tilde{\otimes}(\theta\tilde{\cdot}\bnb q)\no\\
  &+(\bnb^2\theta)\nb_Nq+4(\bnb\theta)\tilde{\otimes}\bnb\nb_Nq+6\theta\tilde{\otimes}\bnb^2\nb_N q\no\\
  &-3\theta\tilde{\otimes}(\theta\tilde{\cdot}\theta)\nb_N q+3\theta\tilde{\otimes}\theta \nb_N^2q.\label{eq.CL004.22}
\end{align}

\begin{definition}\label{defn.3.3}
  Let $\N(\bar{x})$ be the outward unit normal to $\D\dm$ at $\bar{x}\in \D\dm$. Let $\dist(x_1,x_2)=|x_1-x_2|$ denote the Euclidean distance in $\R^n$, and for $\bar{x}_1, \bar{x}_2\in \D\dm$, let $\dist_{\D\dm} (\bar{x}_1, \bar{x}_2)$ denote the geodesic distance on the boundary.
\end{definition}

\begin{definition}\label{defn.3.4}
  Let $\dist(x,\D\dm)$ be the Euclidean distance from $x$ to the boundary. Let $\iota_0$ be the injectivity radius of the normal exponential map of $\D\dm$, i.e., the largest number such that the map
\begin{align}
  \begin{aligned}
    \D\dm\times (-\iota_0,\iota_0)&\to\{x\in\R^n: \dist(x,\D\dm)<\iota\}\\
    \text{given by } (\bar{x},\iota)&\to x=\bar{x}+\iota \N(\bar{x})
  \end{aligned}
\end{align}
is an injection.
\end{definition}

\begin{definition}\label{defn.3.5}
  Let $0<\eps_1<2$ be a fixed number, and let $\iota_1=\iota_1(\eps_1)$ the largest number such that
  \begin{align}
    \abs{\N(\bar{x}_1)-\N(\bar{x}_2)}\ls \eps_1 \quad \text{whenever } \abs{\bar{x}_1-\bar{x}_2}\ls \iota_1, \; \bar{x}_1,\bar{x}_2\in\D\dm.
  \end{align}
\end{definition}

\begin{lemma}[\mbox{\cite[Lemma 3.9]{CL00}}] \label{lem.CL00lem3.9}
  Let $N$ be the unit normal to $\D\Omega$, and let $h_{ab}=\frac{1}{2}D_tg_{ab}$. On $[0,T]\times \D\Omega$, we have
  \begin{align}\label{eq.CL00lem3.9.1}
    &D_tN_a=h_{NN}N_a,\quad D_tN^c=-2h_d^cN^d+h_{NN}N^c,\\
    &D_t\gamma^{ab}=-2\gamma^{ac}h_{cd}\gamma^{db},\label{eq.CL00lem3.9.2}
  \end{align}
  where $h_{NN}=h_{ab}N^aN^b$. The volume element on $\D\Omega$ satisfies
  \begin{align}\label{eq.CL00lem3.9.3}
    D_td\mu_\gamma=(\tr h-h_{NN})d\mu_\gamma=(\tr \theta u\cdot N+\gamma^{ab}\bnb_a\bar{u}_b)d\mu_\gamma,
  \end{align}
  where $\bar{u}_b$ denotes the tangential component of $u_b$ to the boundary $\D\Omega$.
\end{lemma}

\begin{lemma}[\mbox{cf. \cite[Lemma 5.5]{CL00}}] \label{lem.CL00lem5.5}
  Let $w_a=w_{Aa}=\nb_A^r f_a$, $\nb_A^r=\nb_{a_1}\cdots \nb_{a_r}$, $f$ be a $(0,1)$ tensor, and $[\nb_a,\nb_b]=0$. Let $\dv w=\nb_a w^a=\nb^r\dv f$, and let $(\curl w)_{ab}=\nb_aw_b-\nb_b w_a=\nb^r(\curl f)_{ab}$. Then,
  \begin{align}
    |\nb w|^2\ls C(g^{ab}\gamma^{cd}\gamma^{AB}\nb_c w_{Aa}\nb_d w_{Bb}+|\dv w|^2+|\curl w|^2).
  \end{align}
\end{lemma}

\begin{lemma}[\mbox{\cite[Proposition 5.8]{CL00}}] \label{lem.CL00prop5.8}
Let $\iota_0$ and $\iota_1$ be as in Definitions \ref{defn.3.4} and \ref{defn.3.5}, and suppose that $|\theta|+1/\iota_0\ls K$ and $1/\iota_1\ls K_1$. Then with $\tilde{K}=\min(K,K_1)$ we have, for any $r\gs 2$ and $\delta>0$,
\begin{align}
  &\norm{\nb^r q}_{L^2(\D\Omega)}+\norm{\nb^r q}_{L^2(\Omega)}\no\\
  &\qquad\ls C\norm{\Pi \nb^r q}_{L^2(\D\Omega)}+C(\tilde{K},\vol\Omega)\sum_{s\ls r-1} \norm{\nb^s\Delta q}_{L^2(\Omega)},\label{eq.CL00prop5.8.1}\\
  &\norm{\nb^{r-1} q}_{L^2(\D\Omega)}+\norm{\nb^r q}_{L^2(\Omega)}\no\\
  &\qquad\ls \delta\norm{\Pi \nb^r q}_{L^2(\D\Omega)}+C(1/\delta,K,\vol\Omega)\sum_{s\ls r-2} \norm{\nb^s\Delta q}_{L^2(\Omega)}.\label{eq.CL00prop5.8.2}
\end{align}
\end{lemma}

\begin{lemma}[cf. \mbox{\cite[Proposition 5.9]{CL00}}] \label{lem.CL00prop5.9}
  Assume that $0\ls r\ls 4$. Suppose that $|\theta|\ls K$ and $\iota_1\gs 1/K_1$, where $\iota_1$ is as in Definition 3.5 of \cite{CL00}. If $q=0$ on $\D\Omega$, then for $m=0,1$,
  \begin{align}\label{eq.CL00prop5.9.1}
    \norm{\Pi\nb^r q}_{L^2(\D\Omega)}\ls & C(K,K_1)\left(\norm{\theta}_{L^\infty(\D\Omega)}+\sum_{k\ls r-2-m} \norm{\bnb^k \theta}_{L^2(\D\Omega)}\right)\no \\
    &\times\sum_{k\ls r-2+m}\norm{\nb^k q}_{L^2(\D\Omega)}.
  \end{align}
  If, in addition, $|\nb_N q|\gs \eps>0$ and $|\nb_N q|\gs 2\eps \norm{\nb_N q}_{L^\infty(\D\Omega)}$, then
  \begin{align}\label{eq.CL00prop5.9.2}
    &\norm{\bnb^{r-2}\theta}_{L^2(\D\Omega)} \ls\\
    &\qquad C\left(K,K_1,\frac{1}{\eps}\right) \left(\norm{\theta}_{L^\infty(\D\Omega)}+\sum_{k\ls r-3} \norm{\bnb^k \theta}_{L^2(\D\Omega)}\right)\sum_{k\ls r-1} \norm{\nb^k q}_{L^2(\D\Omega)}.\no
  \end{align}
\end{lemma}

\begin{lemma}[cf. \mbox{\cite[Proposition 5.10]{CL00}}] \label{lem.CL00prop5.10}
  Assume that $0\ls r\ls 4$ and that $|\theta|+1/\iota_0\ls K$. If $q=0$ on $\D\Omega$, then
  \begin{align}
    \norm{\nb^{r-1} q}_{L^2(\D\Omega)} \ls& C\left(\norm{\bnb^{r-3}\theta}_{L^2(\D\Omega)} \norm{\nb_N q}_{L^\infty(\D\Omega)}+\norm{\nb^{r-2}\Delta q}_{L^2(\Omega)}\right)\no\\
    &+C\left(K,\vol\Omega,\norm{\theta}_{L^2(\D\Omega)}\right)\no\\
    &\qquad\times\left(\norm{\nb_N q}_{L^\infty(\D\Omega)}+\sum_{s=0}^{r-3} \norm{\nb^s\Delta q}_{L^2(\Omega)}\right).
  \end{align}
\end{lemma}


\begin{lemma}[\mbox{\cite[Lemma A.1]{CL00}}] \label{lem.CL00lemA.1}
    If $\alpha$ is a $(0,r)$ tensor, then with $a=k/m$ and a constant $C$ that only depends on $m$ and $n$, such that
  \begin{align}
    \norm{\bnb^k\alpha}_{L^s(\D\Omega)}\ls C\norm{\alpha}_{L^q(\D\Omega)}^{1-a}\norm{\bnb^m \alpha}_{L^p(\D\Omega)}^a,
  \end{align}
  if
   \begin{align*}
     \frac{m}{s}=\frac{k}{p}+\frac{m-k}{q}, \quad 2\ls p\ls s\ls q\ls \infty.
   \end{align*}
\end{lemma}

\begin{lemma}[\mbox{\cite[Lemma A.2]{CL00}}] \label{lem.CL00lemA.2}
  Suppose that for $\iota_1\gs 1/K_1$
  \begin{align}
    \abs{\N(\bar{x}_1)-\N(\bar{x}_2)}\ls \eps_1, \quad \text{whenever } |\bar{x}_1-\bar{x}_2|\ls \iota_1, \; \bar{x}_1,\bar{x}_2\in\D\dm,
  \end{align}
  and
  \begin{align}
    C_0^{-1}\gamma_{ab}^0(y) Z^aZ^b\ls \gamma_{ab}(t,y)Z^aZ^b\ls C_0\gamma_{ab}^0(y) Z^aZ^b, \quad \text{if } Z\in T(\Omega),
  \end{align}
  where $\gamma_{ab}^0(y)=\gamma_{ab}(0,y)$. Then if $\alpha$ is a $(0,r)$ tensor,
  \begin{align}
    &\norm{\alpha}_{L^{(n-1)p/(n-1-kp)}(\D\Omega)}\ls C(K_1) \sum_{\ell=0}^k \norm{\nb^\ell \alpha}_{L^p(\D\Omega)}, \quad 1\ls p<\frac{n-1}{k},\\
    &\norm{\alpha}_{L^\infty(\D\Omega)}\ls \delta\norm{\nb^k \alpha}_{L^p(\D\Omega)}+C_\delta(K_1)\sum_{\ell=0}^{k-1} \norm{\nb^\ell \alpha}_{L^p(\D\Omega)}, \quad k>\frac{n-1}{p},\label{eq.A.32}
  \end{align}
  for any $\delta>0$.
\end{lemma}

\begin{lemma}[\mbox{\cite[Lemma A.3]{CL00}}] \label{lem.CL00lemA.3}
  With notation as in Lemmas \ref{lem.CL00lemA.1} and \ref{lem.CL00lemA.2}, we have
  \begin{align}
    \sum_{j=0}^k\norm{\nb^j\alpha}_{L^s(\Omega)}\ls C\norm{\alpha}_{L^q(\Omega)}^{1-a}\left(\sum_{i=0}^m \norm{\nb^i\alpha}_{L^p(\Omega)}K_1^{m-i}\right)^a.
  \end{align}
\end{lemma}

\begin{lemma}[\mbox{\cite[Lemma A.4]{CL00}}] \label{lem.CL00lemA.4}
  Suppose that $\iota_1\gs 1/K_1$ and $\alpha$ is a $(0,r)$ tensor. Then
  \begin{align}
    \norm{\alpha}_{L^{np/(n-kp)}(\Omega)} \ls& C\sum_{\ell=0}^k K_1^{k-\ell} \norm{\nb^\ell \alpha}_{L^p(\Omega)}, \quad 1\ls p<\frac{n}{k},\label{eq.A.4.1}\\
    \norm{\alpha}_{L^\infty(\Omega)}\ls &C\sum_{\ell=0}^k K_1^{n/p-\ell} \norm{\nb^\ell \alpha}_{L^p(\Omega)}, \quad k>\frac{n}{p}.\label{eq.A.4.2}
  \end{align}
\end{lemma}

\begin{lemma}[\mbox{\cite[Lemma A.5]{CL00}}] \label{lem.CL00lemA.5}
  Suppose that $q=0$ on $\D\Omega$. Then
  \begin{align}
    \norm{q}_{L^2(\Omega)}\ls&C(\vol\Omega)^{1/n}\norm{\nb q}_{L^2(\Omega)},\\
    \norm{\nb q}_{L^2(\Omega)}\ls&C(\vol\Omega)^{1/2n}\norm{\Delta q}_{L^2(\Omega)}.
  \end{align}
\end{lemma}

\begin{lemma}[\mbox{\cite[Lemma A.7]{CL00}}] \label{lem.CL00lemA.7}
  Let $\alpha$ be a $(0,r)$ tensor. Assume that
  $$\vol\Omega \ls V \text{ and  }\norm{\theta}_{L^\infty(\D\Omega)}+1/\iota_0 \ls K,$$
  then there is a $C=C(K,V,r,n)$ such that
  \begin{align}
    &\norm{\alpha}_{L^{(n-1)p/(n-p)}(\D\Omega)} \ls C\norm{\nb \alpha}_{L^p(\Omega)} +C\norm{\alpha}_{L^p(\Omega)},\quad 1\ls p<n,\label{eq.CL00lemA.7.1}\\
    &\norm{\nb^2\alpha}_{L^2(\Omega)} \ls C\left(\norm{\Pi\nb^2\alpha}_{L^{2(n-1)/n}(\D\Omega)} +\norm{\Delta\alpha}_{L^2(\Omega)}+\norm{\nb\alpha}_{L^2(\Omega)}\right). \label{eq.CL00lemA.7.2}
  \end{align}
\end{lemma}


\end{document}